\documentclass[12pt]{amsart}
\usepackage{a4wide}
\usepackage{amssymb}
\usepackage{amsmath}
\usepackage{amsfonts}
\usepackage{amsthm}
\usepackage{graphicx}
\usepackage{epsfig}
\usepackage{longtable}
\usepackage{paralist}
\usepackage[usenames,dvipsnames]{color}
\usepackage{hyperref}
\usepackage{caption}
\usepackage[labelformat=simple,labelfont={}]{subfig}

\usepackage{xcolor}
\usepackage[usenames,dvipsnames]{color}
\usepackage[inline]{enumitem}

\newtheorem{thm}{Theorem}
\newtheorem{lem}[thm]{Lemma}
\newtheorem{cor}[thm]{Corollary}
\newtheorem{prop}[thm]{Proposition}

\theoremstyle{definition}

\theoremstyle{remark}
\newtheorem*{rmk}{Remark}

\newcommand{\eps}{\varepsilon}

\newcommand{\DEF}{{:=}}

\newcommand{\PT}[1]{\mathbf{#1}}

\newcommand{\re}{\mathop{\mathrm{Re}}}

\DeclareMathOperator{\dd}{\mathrm{d}}

\DeclareMathOperator{\gammafcn}{\Gamma}

\DeclareMathOperator{\RamanujanL}{L}

\DeclareMathOperator{\zetafcn}{\zeta}

\DeclareMathOperator{\HyperF}{F}

\newcommand{\Hypergeom}[5]{{\sideset{_#1}{_#2}\HyperF\!\left(\substack{\displaystyle#3\\\displaystyle#4};#5\right)}}

\newcommand{\Pochhsymb}[2]{{\left(#1\right)_{#2}}}

\allowdisplaybreaks[1]

\title{On the lower bounds for the spherical cap discrepancy}
\author[D. Bilyk and J. S. Brauchart]{Dmitriy Bilyk\textdagger and Johann S. Brauchart\textasteriskcentered{}} 
\thanks{\noindent 
\textasteriskcentered Corresponding author.  \\ 
\textdagger The research of this author was supported by  Fulbright Austria and NAWI Graz. }

\date{\today}

\DeclareCaptionLabelFormat{andtable}{#1˜\ #2 and \tablename˜\ \thetable}

\begin{document}

\address{J. S. Brauchart: 
Institute of Analysis and Number Theory, 
Graz University of Technology, 
Kopernikusgasse 24/II, 
8010 Graz, 
Austria}
\email{j.brauchart@tugraz.at}

\address{D. Bilyk:
 School of Mathematics,
 University of Minnesota, 
 206 Church St. SE
 Minneapolis, MN, 55408}
\email{dbilyk@umn.edu}

\begin{abstract}  We start by providing  a very simple and elementary new proof of the classical bound due to J. Beck which states that the spherical cap $\mathbb{L}_2$-discrepancy of any $N$ points on the unit sphere $\mathbb S^d$ in $\mathbb{R}^{d+1}$, $d\geq2$, is at least of the order $N^{-\frac12-\frac{1}{2d}}$. The argument used in this proof leads us to many further new results: estimates of the discrepancy in terms of various geometric quantities, an easy proof of {point-independent} upper estimates for the sum of positive powers of  Euclidean distances between points on the sphere, lower bounds for the discrepancy of rectifiable curves and sets of arbitrary Hausdorff dimension.
Moreover, refinements of the proof also allow us to obtain  explicit values of the constants in the lower discrepancy  bound on $\mathbb{S}^d$. The value of the obtained  asymptotic constant   falls within $3\%$    of the conjectured optimal constant  on $\mathbb S^2$ (and within up to $7\%$ on $\mathbb S^4$, $\mathbb S^8$, $\mathbb S^{24}$). 
\end{abstract}


\maketitle


\section{Introduction}

The \emph{spherical cap $\mathbb{L}_2$-discrepancy}
\begin{equation} \label{eq:discdef}
D_{\mathbb{L}_2}^{\mathrm{C}}( \PT{x}_1, \dots, \PT{x}_N ) \DEF \sqrt{\int_{-1}^1 \int_{\mathbb{S}^d} \left| \frac{\big| \big\{ k : \PT{x}_{k} \in C( \PT{x}, t ) \big\} \big|}{N} - \sigma_d( C( \PT{x}, t ) ) \right|^2 \dd \sigma_d( \PT{x} ) \dd t}
\end{equation}
measures the irregularity of $N$ points $\PT{x}_1, \dots, \PT{x}_N$ on the unit sphere~$\mathbb{S}^d$ in the Euclidean space $\mathbb{R}^{d+1}$, $d \geq 2$, with respect to spherical caps $C( \PT{x}, t ) \DEF \{ \PT{z} \in \mathbb{S}^d :  \PT{z} \cdot \PT{x} \geq t \}$ in the $\mathbb{L}_2$-sense. Here, $\sigma_d$ denotes the normalized uniform surface area measure on $\mathbb{S}^d$. In other words, $D_{\mathbb{L}_2}^{\mathrm{C}}$ is the quadratic average of the deviation of the discrete empirical measure generated by the points $\PT{x}_1, \dots, \PT{x}_N$ from the uniform measure. 

The main postulate of the theory of irregularity of distributions states that discrete point configurations  cannot be `too uniform'. In the case of spherical cap $\mathbb{L}_2$-discrepancy, this is manifested by the following  lower bound  due to Beck \cite{Be1984}, which is uniform for $N\ge 2$:  there exists a constant $c_d >0$ such that for all $\PT{x}_1, \dots, \PT{x}_N  \in \mathbb S^d$, the discrepancy satisfies
\begin{equation} \label{eq:Beck}
D_{\mathbb{L}_2}^{\mathrm{C}}( \PT{x}_1, \dots, \PT{x}_N ) \ge c_d N^{-\frac12-\frac{1}{2d}}.
\end{equation}
This bound provides the optimal order of the spherical cap $\mathbb{L}_2$-discrepancy in terms of the number of points $N$ as can be seen from Stolarsky's invariance principle~ \cite{St1973} (cf. \eqref{eq:Stolarsky.s.invariance.principle} below). Beck's proof of \eqref{eq:Beck}  relied on the Fourier transform method, and an alternative proof given in \cite{BiDa2019} relied on spherical harmonics and Gegenbauer polynomial  expansions.


In Section~\ref{sec:simple.proof} of this paper we provide a new and simple proof of Beck's lower bound~\eqref{eq:Beck}. In addition to being completely elementary, this argument is powerful enough to allow for a variety of extensions, corollaries, and generalizations. 

In particular, in Section~\ref{sec:geometric} we provide new lower bounds for  the discrepancy $D_{\mathbb{L}_2}^{\mathrm{C}}$ in terms of  geometric quantities such as the norm of the centroid and the frame potential.

In Section~\ref{sec:dimension} we prove analogs of Beck's bound for the discrepancy of sets of arbitrary Hausdorff dimension $0\le s \le d$ (rather than just point sets). This question did not appear to be accessible with any of the previous methods. 

Moreover, in Section~\ref{sec:constants} we carefully analyze and refine the argument employed in Section~\ref{sec:simple.proof} to obtain explicit constants in Beck's lower bound \eqref{eq:Beck} which are universally or asymptotically valid. The asymptotic constant is surprisingly tight in the sense of the conjectured behaviour for $\mathbb{S}^2$ (cf. \cite{Br2011}) as we observe that the asymptotical constant recovers one significant digit (ca. $97\%$) of the factor of the leading term in the asymptotics of the minimal discrepancy.  Similarly, this analysis recovers $96\%$, $95\%$, and $93\%$ of the  conjectured optimal constant on $\mathbb S^4$, $\mathbb S^8$, and $\mathbb S^{24}$ respectively, see Table \ref{tbl:summary}. 
%
%

In Section~\ref{sec:powers.and.generalizations} we apply the same argument as in Section~\ref{sec:simple.proof} to obtain a much simpler proof of the result  of \cite{Wa1990:lower} on the lower bound for the difference between the sum and integral of powers of Euclidean distances for the points on the sphere. {We compare the constant in our bound with the conjectured constant of the asymptotic expansion (Figure~\ref{fig:calphad}).} In addition, we obtain a general result (Proposition~\ref{prop:wagner+}) providing lower bounds on the difference of discrete and continuous energies in terms of the decay of the Taylor coefficients of the kernel.


To finish the introduction, we would like to note that Beck's lower bound \eqref{eq:Beck} is tight. 
Indeed, while the `typical' discrepancy of $N$ independently and identically randomly chosen points $\PT{X}_1, \dots, \PT{X}_N$ on $\mathbb{S}^d$ has order $N^{-\frac{1}{2}}$ as can be seen from
\begin{equation*}
\mathbb{E}\big[(D_{\mathbb{L}_2}^{\mathrm{C}}( \PT{X}_1, \dots, \PT{X}_N ))^2\big] = \left( \frac{1}{d} \, \frac{\gammafcn( \frac{d+1}{2} )}{ \sqrt{\pi} \, \gammafcn( \frac{d}{2} )} \int_{\mathbb{S}^d} \int_{\mathbb{S}^d} \left\| \PT{x} - \PT{y} \right\| \dd \sigma_d( \PT{x} ) \dd \sigma_d( \PT{y} ) \right) \frac{1}{N}
\end{equation*}
(this easily follows from  Stolarsky's invariance principle presented in \eqref{eq:Stolarsky.s.invariance.principle} below), a  more restrictive random model known as jittered sampling (cf. \cite{BrSaSlWo2014}) recovers the optimal order $N^{-\frac{1}{2}-\frac{1}{2d}}$  as in \eqref{eq:Beck}. 
{As to other random point processes, the precise asymptotic behaviour with optimal order is known for the spherical ensembles on $\mathbb{S}^2$ (\cite{AlZa2015}), whereas for the harmonic ensemble on $\mathbb{S}^d$ the precise rate is $N^{-\frac{1}{2}-\frac{1}{2d}} (\log N)^{\frac{1}{2}}$ (\cite{BoGrMa2024}).} It should be mentioned that $N^{-\frac{1}{2}}$ is the best proven rate of an upper bound of the discrepancy for explicitly constructed (deterministic) $N$-point configurations on $\mathbb{S}^2$ (cf.~\cite{AiBrDi2012,Et2021}).

\section{A simple proof of Beck's lower bound}
\label{sec:simple.proof}

In this note we give a short proof of the lower bound \eqref{eq:Beck} which is completely elementary in nature, only uses basic calculus, some asymptotic analysis, and does not appeal to Fourier analysis. The presentation is essentially self-contained. The argument also allows for obtaining explicit values of the constant $c_d$ in \eqref{eq:Beck} and extends to more general settings including positive powers of the distances between $N$ points on the sphere.


\emph{Preliminaries:} We shall use the Pochhammer symbol to denote rising factorials
\begin{equation*}
\Pochhsymb{a}{0} := 1, \qquad \Pochhsymb{a}{n+1} \DEF \left( a + n \right) \Pochhsymb{a}{n}, \quad n\in \mathbb{N}_0,
\end{equation*}
which satisfy the following identities (\cite[Section~5.2(iii)]{NIST:DLMF})
\begin{equation} \label{eq:Pochhammer.identities}
\binom{a}{n} = (-1)^n \frac{\Pochhsymb{-a}{n}}{n!}, \qquad \Pochhsymb{a}{n} = \frac{\gammafcn( a + n )}{\gammafcn( a )}, \qquad \Pochhsymb{a}{2n} = 2^{2n} \Pochhsymb{\frac{a}{2}}{n} \Pochhsymb{\frac{a+1}{2}}{n}.
\end{equation}
We also need the asymptotic behaviour of a ratio of gamma functions (\cite[Section~5.11(iii)]{NIST:DLMF}),
\begin{equation*}
\frac{\gammafcn(n + a)}{\gammafcn( n + b )} \sim n^{a-b} \qquad \text{as $n \to \infty$ in $\mathbb{N}$.}
\end{equation*}
Furthermore, we observe that the power series expansion 
\begin{equation*}
\left( 1 - t \right)^{\frac{1}{2}} = \sum_{m=0}^\infty (-1)^m \binom{\frac{1}{2}}{m} t^m 
\end{equation*}
converges absolutely in $(-1,1)$ and at the endpoints $\pm1$ because of 
\begin{equation}\label{eq:coef}
(-1)^m \binom{\frac{1}{2}}{m} = \frac{\Pochhsymb{-\frac{1}{2}}{m}}{m!} = \frac{1}{\gammafcn( - \frac{1}{2} )} \, \frac{\gammafcn( m - \frac{1}{2} )}{\gammafcn( m + 1 )} \sim -\frac{1}{2 \sqrt{\pi }} \, \frac{1}{m^{\frac{3}{2}}} \qquad \text{as $m \to \infty$.}
\end{equation}
Thus, using the fact that 
\begin{equation*}
\left\| \PT{x} - \PT{y} \right\|^2 = \left( \PT{x} - \PT{y} \right) \cdot \left( \PT{x} - \PT{y} \right) = 2 - 2 \PT{x} \cdot \PT{y}, \qquad \PT{x}, \PT{y} \in \mathbb{S}^d,
\end{equation*}
we get an absolutely converging series expansion of the distance
\begin{equation}
\label{eq:series.expansion.distance}
\left\| \PT{x} - \PT{y} \right\| = \sqrt{2} \, \sum_{m=0}^\infty (-1)^m \binom{\frac{1}{2}}{m} \left( \PT{x} \cdot \PT{y} \right)^m, \qquad \PT{x}, \PT{y} \in \mathbb{S}^d.
\end{equation}

The connection between discrepancy and sum of distances is provided by {\emph{Stolarsky's invariance principle}} (cf. \cite{St1973, BrDi2013,BiDaMa2018}\footnote{We would like to remark that the proof of \eqref{eq:Stolarsky.s.invariance.principle}  given in \cite{BiDaMa2018} is completely elementary, thus making the entire proof of Beck's bound \eqref{eq:Beck} presented here elementary.}): for all $\PT{x}_1, \dots, \PT{x}_N  \in \mathbb S^d$
\begin{equation} \label{eq:Stolarsky.s.invariance.principle}
\frac{1}{C_d} \left( D_{\mathbb{L}_2}^{\mathrm{C}}( \PT{x}_1, \dots, \PT{x}_N ) \right)^2 
=
\int_{\mathbb{S}^d} \int_{\mathbb{S}^d} \left\| \PT{x} - \PT{y} \right\| \dd \sigma_d( \PT{x} ) \dd \sigma_d( \PT{y} ) - \frac{1}{N^2} \sum_{j=1}^N \sum_{k=1}^N \left\| \PT{x}_j - \PT{x}_k \right\|.
\end{equation}
This shows that minimizing the spherical cap discrepancy is equivalent to maximizing the sum of  Euclidean distances. 

The constant $C_d$ is expressed in terms of the surface area $\omega_d$ of $\mathbb{S}^d$ or gamma functions as
\begin{equation}\label{eq:Cd}
C_2 := \frac{1}{4}, \qquad C_d \DEF \frac{1}{d} \, \frac{\omega_{d-1}}{\omega_d} = \frac{1}{d} \, \frac{\gammafcn( \frac{d+1}{2} )}{ \sqrt{\pi} \, \gammafcn( \frac{d}{2} )}, \quad d\geq 2.
\end{equation}
{It represents the ratio of the volume of the unit ball in $\mathbb{R}^d$ and the surface area measure of the unit sphere in $\mathbb{R}^{d+1}$.}

%
%
Substituting the series expansion of the distance \eqref{eq:series.expansion.distance} into \eqref{eq:Stolarsky.s.invariance.principle} and rearranging the terms, we obtain
\begin{equation} \label{eq:discrepancy.formula.02}
\begin{split}
&\left( D_{\mathbb{L}_2}^{\mathrm{C}}( \PT{x}_1, \dots, \PT{x}_N ) \right)^2 \\
&\phantom{eq}=
 \sqrt{2} \, C_d \sum_{m=1}^\infty \frac{-\Pochhsymb{-\frac{1}{2}}{m}}{m!} \left( \frac{1}{N^2} \sum_{j=1}^N \sum_{k=1}^N \left( \PT{x}_j \cdot \PT{x}_k \right)^m - \int_{\mathbb{S}^d} \int_{\mathbb{S}^d} \left( \PT{x} \cdot \PT{y} \right)^m  \dd \sigma_d( \PT{x} ) \dd \sigma_d( \PT{y} ) \right).
\end{split}
\end{equation}
A crucial observation is that the coefficients $C_d \frac{-\Pochhsymb{-\frac{1}{2}}{m}}{m!}$, $m \geq 1$, are all positive. The rest of the proof of \eqref{eq:Beck} relies on two simple lemmas, whose proofs are postponed to  Appendix~\ref{app:lem}.

The first lemma shows that all of the parenthetical expressions in \eqref{eq:discrepancy.formula.02} are non-negative. 
\begin{lem}\label{lem:1st.non.negativity}
Let $d \geq 2$. 
For all $\PT{x}_1, \dots, \PT{x}_N  \in \mathbb S^d$ and for any $m \ge 0$, 
\begin{equation}\label{eq:pos}
 \frac{1}{N^2} \sum_{j=1}^N \sum_{k=1}^N \left( \PT{x}_j \cdot \PT{x}_k \right)^m - \int_{\mathbb{S}^d} \int_{\mathbb{S}^d} \left( \PT{x} \cdot \PT{y} \right)^m  \dd \sigma_d( \PT{x} ) \dd \sigma_d( \PT{y} ) \ge 0.
\end{equation}
\end{lem}
This result  can be viewed as a real-valued version of the classical Welch bounds \cite{Welch}  and it can be  traced back to \cite{Si}. It  follows easily from a  well-known fact in classical potential theory stating that for positive definite kernels on the sphere, the associated energy is minimized by the uniform surface area measure, see e.g. \cite{La1972, BoHaSaBook2019}, but for the sake of completeness we provide a simple self-contained proof in Appendix~\ref{app:lem}. 

The left-hand side of \eqref{eq:pos} vanishes for all spherical $m$-designs (i.e., such  configurations $\PT{x}_1, \dots, \PT{x}_N  \in \mathbb S^d$  that $\frac{1}{N} \sum_{j=1}^N p (\PT{x}_j) = \int_{\mathbb S^d} p (\PT{x} ) \dd \sigma_d( \PT{x} )$ for all polynomials of degree at most~$m$\textcolor{olive}{; cf.}~\cite{BoRaVi2013}).  A full characterization of all sets for which the left-hand side is zero is slightly more complicated, and we do not discuss it in detail (when  $m=1$ and $m=2$, it is discussed in Section \ref{sec:geometric}). Observe that  
 for odd $m$, one obtains zero for any centrally symmetric (antipodal) configuration.  \\


The second lemma provides the value and the asymptotics of the integrals in \eqref{eq:discrepancy.formula.02}.
\begin{lem}\label{lem:2nd.inner.product.integrals.asymptotics}
Let $d \geq 2$. 
If $m=2r-1$ is odd, then $\displaystyle{ \int_{\mathbb{S}^d} \int_{\mathbb{S}^d} \left( \PT{x} \cdot \PT{y} \right)^m  \dd \sigma_d( \PT{x} ) \dd \sigma_d( \PT{y} )  = 0}$. 

\noindent If $m= 2r$ is even, then 
\begin{equation} \label{eq:lem.1.id.2}
\int_{\mathbb{S}^d} \int_{\mathbb{S}^d} \left( \PT{x} \cdot \PT{y} \right)^m  \dd \sigma_d( \PT{x} ) \dd \sigma_d( \PT{y} )   =  \frac{\Pochhsymb{\frac{1}{2}}{r}}{\Pochhsymb{\frac{d+1}{2}}{r}} \sim
\frac{\gammafcn( \frac{d+1}{2} )}{\gammafcn( \frac{1}{2} )} \, \frac{1}{r^{\frac{d}{2}}} \qquad \text{as $r \to \infty$}
\end{equation}
and the asymptotic term is a strict upper bound for all $r\geq 1$. 
\end{lem}

Assuming these two lemmas,  we are ready to give a new simple  proof of Beck's lower bound~\eqref{eq:Beck}. 
Utilizing  non-negativity of each term (due to Lemma~\ref{lem:1st.non.negativity}) in relation~\eqref{eq:discrepancy.formula.02}, we omit all terms with odd integers $m$ and consider only even terms with $m =2r \geq 2M$ for some $M \in \mathbb N$ to be chosen later. Then
\begin{equation} \label{eq:discrepancy.formula.lower.bound.01}
\left( D_{\mathbb{L}_2}^{\mathrm{C}}( \PT{x}_1, \dots, \PT{x}_N ) \right)^2 
\geq
\sqrt{2} \, C_d \sum_{r=M}^\infty \frac{-\Pochhsymb{-\frac{1}{2}}{2r}}{(2r)!} \left( \frac{1}{N^2} \sum_{j=1}^N \sum_{k=1}^N \left( \PT{x}_j \cdot \PT{x}_k \right)^{2r} - \frac{\Pochhsymb{\frac{1}{2}}{r}}{\Pochhsymb{\frac{d+1}{2}}{r}} \right).
\end{equation}

By Lemma~\ref{lem:2nd.inner.product.integrals.asymptotics}, the coefficients $\frac{\Pochhsymb{\frac{1}{2}}{r}}{\Pochhsymb{\frac{d+1}{2}}{r}}$ tend to zero with decay rate $r^{-\frac{d}{2}}$. Thus, using only asymptotic reasoning, we may estimate for $ r \geq M$,
\begin{equation} \label{eq:2N}
\frac{1}{N^2} \sum_{j=1}^N \sum_{k=1}^N \left( \PT{x}_j \cdot \PT{x}_k \right)^{2r} - \frac{\Pochhsymb{\frac{1}{2}}{r}}{\Pochhsymb{\frac{d+1}{2}}{r}} 
\geq 
\frac{1}{N} - \frac{c}{M^{d/2}} 
\geq
\frac{1}{N} - \frac{1}{2N} 
= 
\frac{1}{2N},  
\end{equation}
where in the first step we keep only the diagonal terms (i.e. $j = k$) in the double sum for the points and in the second step we require that  
$\frac{c}{M^{d/2}}< \frac{1}{2N}$ for all $r\geq M$ which is guaranteed when $N$ and $M$ are related by means of 
\begin{equation} \label{eq:M.N.relation}
\alpha \leq \frac{M}{N^{\frac{2}{d}}} \leq \alpha + 1, \qquad N \geq N_0,
\end{equation}
for some sufficiently large $\alpha$ depending only on $d$ and $N_0$.

Finally, as shown in \eqref{eq:coef}, the coefficients $\frac{-\Pochhsymb{-\frac{1}{2}}{2r}}{(2r)!} $ are of the order $r^{-\frac{3}{2}}$, and therefore 
\begin{equation}\label{eq:sum}
\sum_{r=M}^\infty \frac{-\Pochhsymb{-\frac{1}{2}}{2r}}{(2r)!} \ge c' \sum_{r=M}^\infty  r^{-\frac32} \ge 2c' M^{-\frac12} \ge c''_d N^{-\frac{1}{d}}.
\end{equation}
The last relation follows from \eqref{eq:M.N.relation}. 

Putting together \eqref{eq:discrepancy.formula.lower.bound.01}, \eqref{eq:2N}, and \eqref{eq:sum}, we find that 
\begin{equation*}
\left( D_{\mathbb{L}_2}^{\mathrm{C}}( \PT{x}_1, \dots, \PT{x}_N ) \right)^2 
\geq
\sqrt{2} \, C_d \, \frac{1}{2N} \, \sum_{r=M}^\infty \frac{-\Pochhsymb{-\frac{1}{2}}{2r}}{(2r)!} 
\geq 
c_d^2 \,  N^{-1- \frac{1}{d}},
\end{equation*}
which finishes the proof of Beck's lower bound \eqref{eq:Beck}.

\section{Geometric corollaries}\label{sec:geometric}

The method of proof presented in Section \ref{sec:simple.proof} allows  {us} to obtain some immediate  corollaries which provide a lower bound for the discrepancy $D_{\mathbb{L}_2}^{\mathrm{C}}$ in terms of quantities which have geometric interpretations. Recall  that every term in the right-hand side of \eqref{eq:discrepancy.formula.02} is non-negative as proved in Lemma \ref{lem:1st.non.negativity}.  Therefore,  discrepancy can be bounded from below by any of these terms. In particular, taking $m=1$ and observing that 
\begin{equation*}
\int_{\mathbb{S}^d} \int_{\mathbb{S}^d} \left( \PT{x} \cdot \PT{y} \right)^m  \dd \sigma_d( \PT{x} ) \dd \sigma_d( \PT{y} )
= 0 \quad \quad \textup{ and } \quad \quad  \frac{1}{N^2} \sum_{j,k=1}^N   \PT{x}_j \cdot \PT{x}_k =  \left\| \frac{1}{N} \sum_{j=1}^N \PT{x}_j \right\|^2,
 \end{equation*}
 we obtain the following bound on the discrepancy in terms of the center of mass of the points. 

\begin{prop}\label{prop:m1} For any $N\ge 2$ and $ \PT{x}_1, \dots, \PT{x}_N \in \mathbb{S}^d$
\begin{equation}\label{eq:m1}
D_{\mathbb{L}_2}^{\mathrm{C}}( \PT{x}_1, \dots, \PT{x}_N ) \geq \sqrt{\frac{C_d}{\sqrt{2}}} \left\| \frac{1}{N} \sum_{j=1}^N \PT{x}_j \right\|. 
\end{equation}
\end{prop} 

\begin{rmk}
In plain words, the conclusion of Proposition \ref{prop:m1} states that if the points $ \PT{x}_1, \dots, \PT{x}_N$ are not well-balanced, i.e. their center of mass is far from zero, then this configuration cannot be distributed too uniformly, i.e. its discrepancy has to  be large. This statement seems to be completely clear heuristically, but we have not found any prior  quantitative  versions in the literature.  

The norm of the centroid of the points vanishes (and is thus   minimized)  for balanced point configurations, i.e.  spherical $1$-designs, which include, in particular, all symmetric (antipodal) configurations.  Of course, in this case \eqref{eq:m1} does not provide meaningful information about the discrepancy. 
\end{rmk}

Similarly, the second term  ($m=2$) of  \eqref{eq:discrepancy.formula.02} provides a  lower bound in terms of the  {\it{frame potentials}} (cf. \cite{BeFi2003}).
\begin{prop}\label{prop:m2} For any $N\ge 2$ and $ \PT{x}_1, \dots, \PT{x}_N \in \mathbb{S}^d$
\begin{equation}\label{eq:m2}
\left( D_{\mathbb{L}_2}^{\mathrm{C}}( \PT{x}_1, \dots, \PT{x}_N ) \right)^2 \geq \frac{C_d}{4\sqrt{2}} \left( \frac{1}{N^2} \sum_{j,k=1}^N ( \PT{x}_j \cdot \PT{x}_k )^2 - \frac{1}{d+1} \right), \qquad \PT{x}_1, \dots, \PT{x}_N \in \mathbb{S}^d.
\end{equation}
\end{prop}

The proof uses the fact that  $$ \int_{\mathbb{S}^d} \int_{\mathbb{S}^d} \left( \PT{x} \cdot \PT{y} \right)^2  \dd \sigma_d( \PT{x} ) \dd \sigma_d( \PT{y} ) = \frac{1}{d+1},$$ which can be seen from Lemma \ref{lem:2nd.inner.product.integrals.asymptotics}.  The quantity $\displaystyle{\sum_{j,k=1}^N ( \PT{x}_j \cdot \PT{x}_k )^2}$ is known as the 
frame potential and was introduced in  \cite{BeFi2003}), where it was proved that it is  minimized   
exactly by the so-called unit norm tight frames (UNTFs)\footnote{The term finite normalized tight frames (FNTFs) has been used in \cite{BeFi2003}, but UNTF appears to be prevalent in the more recent frame theory literature.}, which include, among others, unions of  orthonormal bases and spherical $2$-designs (in fact, $2$-designs are exactly  balanced tight frames), we refer to \cite{BeFi2003,BoRaVi2013} for relevant definitions. 

Thus \eqref{eq:m2} essentially states that if a configuration  $ \PT{x}_1, \dots, \PT{x}_N$ is far from being a tight frame, then its distribution has to be far from uniform. Heuristically, the connection between tight frames and equidsitribution was alluded to in \cite{BeFi2003}, but no quantitative statements akin to \eqref{eq:m2} have been previously known.

More generally, this approach allows one to bound the discrepancy $D_{\mathbb{L}_2}^{\mathrm{C}}$  from below by any term of \eqref{eq:discrepancy.formula.02} with $m\ge 1$. However, the discussion immediately following  Lemma \ref{lem:1st.non.negativity} suggests that for $m\ge 3$ these quantities are not as easy to interpret geometrically, therefore, we do not pursue this direction further in this paper. 

%
%

\section{Discrepancy of curves and  sets of arbitrary Hausdorff dimension.}\label{sec:dimension}

The discrepancy as defined in \eqref{eq:discdef}  measures the quality of approximating the uniform measure $\sigma_d$ on $\mathbb S^d$  by $N$-element point sets, or,  to put it in other words, by $0$-dimensional sets of $0$-dimensional Hausdorff measure $N$. This interpretation suggest a natural generalization of the problem to approximations of $\sigma_d$ by higher dimensional sets (such as curves, surfaces, or fractal sets and sets of fractional dimension). Such questions have been raised, but  to the best of our knowledge almost no results are available in this direction. A recent work \cite{EhGr} focused on one-dimensional (curve) analogs of spherical designs, but there has been no study of discrepancy or energy in this direction. 

Let $0\le s \le d$. For a set $\PT{X} \subset \mathbb S^d$ of Hausdorff dimension $s$  we can define its $s$-dimensional spherical  cap discrepancy as 
\begin{equation}\label{eq:discdef1}
D_{\mathbb{L}_2 ,s}^{\mathrm{C}}( \PT{X} ) \DEF \sqrt{\int_{-1}^1 \int_{\mathbb{S}^d} \left| \frac{\mathcal H_s \big(X \cap C( \PT{x}, t )  \big)}{\mathcal H_s (X)} - \sigma_d( C( \PT{x}, t ) ) \right|^2 \dd \sigma_d( \PT{x} ) \dd t},
\end{equation}
where $\mathcal H_s$ denotes the $s$-dimensional Hausdorff measure. This definition, in complete accordance with \eqref{eq:discdef}, quantifies how well the normalized Hausdorff measure on $X$ approximates the uniform measure $\sigma_d$. Indeed, if $s=0$ and $\PT{X} = \{ \PT{x}_1, \dots, \PT{x}_N \}$, then \eqref{eq:discdef1} coincides with~\eqref{eq:discdef}.

We shall need one more assumption about the regularity of the set $\PT{X}$. We say that an $s$-dimensional  set $\PT{X}$ is lower Alfohrs--David regular (ADR) if there exists a constant $ a(\PT{X},s)$ such that for each $\PT{x} \in \PT{X}$ and for each $\varepsilon \in (0,1)$ the $s$-dimensional Hausdorff measure of the ball $B_\varepsilon (\PT{x}) = \{ \PT{y} \in \mathbb S^d: \, \| \PT{x} - \PT{y} \| < \varepsilon\}$ intersected with $\PT{X}$ satisfies  
\begin{equation}\label{eq:ADR}
\mathcal H_s \big( B_\varepsilon (\PT{x}) \cap \PT{X}  \big) \ge a(\PT{X},s) \varepsilon^s. 
\end{equation}

\begin{thm}\label{thm:sdimensional} 
For any $s\in [0,d)$ there exists a constant $c_{d,s}>0$ such that for any  lower Alfohrs--David regular  set $\PT{X} \subset \mathbb S^d$ of  Hausdorff dimension $s$ with  $\mathcal H_s(\PT{X}) \ge 1$, its discrepancy satisfies 
\begin{equation}\label{eq:sdimensional}
D_{\mathbb{L}_2,s}^{\mathrm{C}}( \PT{X} ) \ge c_{d,s} \left( \frac{  \mathcal H_s (\PT{X}) }{a(\PT{X},s)}\right)^{- \frac{d+1}{2(d-s)}}.
\end{equation}
\end{thm}

\noindent {\emph{Remark:}} Observe that when $\PT{X}$ is an $N$-point configuration (i.e. $s=0$ and $\mathcal H_0 (\PT{X}) = N$), one recovers Beck's lower \eqref{eq:Beck} bound of $N^{-\frac{1}{2}-\frac{1}{2d}}$, while as $s$ approaches $d$ the lower bound approaches zero.  While we don't yet have a proof of sharpness of \eqref{eq:sdimensional}, this tight endpoint behavior leads us to conjecture that \eqref{eq:sdimensional} is sharp for  $0<s <d$.

\begin{proof}
The proof is a variant of the argument in Section \ref{sec:simple.proof}. We start by applying a version of Stolarsky's invariance principle. We define a probability measure $\mu_{\PT{X}}$ supported on $\PT{X}$ by  normalizing the restriction of the  Hausdorff measure $\mathcal H_s$ to $\PT{X}$, i.e.  $\mu_{\PT{X}} = \frac{1}{\mathcal H_s (\PT{X})} \mathcal H_s\vert_{\PT{X}}$. Then we have the following version of  Stolarsky's principle:
\begin{equation} \label{eq:Stolarsky.s.invariance.principle1}
\frac{1}{C_d} \left( D_{\mathbb{L}_2,s}^{\mathrm{C}}( \PT{X} ) \right)^2 
=
\int_{\mathbb{S}^d} \int_{\mathbb{S}^d} \left\| \PT{x} - \PT{y} \right\| \dd \sigma_d( \PT{x} ) \dd \sigma_d( \PT{y} ) -   \int_{\PT{X}} \int_{\PT{X}} \left\| \PT{x} - \PT{y} \right\| \dd \mu_{\PT{X}} ( \PT{x} ) \dd \mu_{\PT{X}} ( \PT{y} ) .
\end{equation}
\noindent For a proof see \cite{BiDaMa2018}, where this identity was actually established for all Borel probability measures on $\mathbb S^d$.  Thus, arguing exactly as in Section \ref{sec:simple.proof} we can obtain the following analogue of \eqref{eq:discrepancy.formula.02}.
\begin{equation} \label{eq:discrepancy.formula.021}
\begin{split}
&\left( D_{\mathbb{L}_2,s}^{\mathrm{C}}( \PT{X} ) \right)^2 \\
&\phantom{eq}=
 \sqrt{2} \, C_d \sum_{m=1}^\infty \frac{-\Pochhsymb{-\frac{1}{2}}{m}}{m!} \left( \int_{\PT{X}} \int_{\PT{X}}  \left( \PT{x} \cdot \PT{y} \right)^m  \dd \mu_{\PT{X}} ( \PT{x} ) \dd \mu_{\PT{X}} ( \PT{y} )   - \int_{\mathbb{S}^d} \int_{\mathbb{S}^d} \left( \PT{x} \cdot \PT{y} \right)^m  \dd \sigma_d( \PT{x} ) \dd \sigma_d( \PT{y} ) \right).
\end{split}
\end{equation}
We observe that all the terms in the series above are non-negative. Indeed, the positivity of the coefficients has already been discussed, see \eqref{eq:discrepancy.formula.02}, and the non-negativity of the expressions in parentheses essentially follows from Lemma \ref{lem:1st.non.negativity} (in fact, the proof presented in Appendix \ref{app:lem} establishes it for all measures). 

The main essential difference in the rest of the proof, as compared to the discrete case,  lies in obtaining lower bounds for the expressions $\int_{\PT{X}} \int_{\PT{X}}  \left( \PT{x} \cdot \PT{y} \right)^{2r} \dd \mu_{\PT{X}} ( \PT{x} ) \dd \mu_{\PT{X}} ( \PT{y} ) $. In \eqref{eq:2N} we estimated analogous  discrete sums by $\frac{1}{N}$ by leaving only the diagonal terms, i.e. considering $x=y$. While this approach will not work for $s\neq 0 $, we emulate it by considering, for each $\PT{x} \in \PT{X}$, only those points $\PT{y} \in \PT{X}$ which lie close to $\PT{x}$, and then invoking lower Alfohrs--David regularity. Let $m=2r$ be even and  $\varepsilon >0$ be a small number. We can then estimate
\begin{align*}
\int_{\PT{X}} \int_{\PT{X}}  \left( \PT{x} \cdot \PT{y} \right)^{2r} \dd \mu_{\PT{X}} ( \PT{x} ) \dd \mu_{\PT{X}} ( \PT{y} )  & = \frac{1}{ \big( \mathcal H_s (\PT{X})\big)^2 } \int_{\PT{X}} \int_{\PT{X}}  \left( \PT{x} \cdot \PT{y} \right)^{2r} \dd \mathcal H_s  ( \PT{x} ) \dd \mathcal H_s ( \PT{y} ) \\
& \ge  \frac{1}{ \big( \mathcal H_s (\PT{X}) \big)^2}   \int_{\PT{X}} \int_{\{ \PT{y} \in \PT{X}: \| \PT{y} - \PT{x}\| < \varepsilon \}}  \left( \PT{x} \cdot \PT{y} \right)^{2r} \dd \mathcal H_s  ( \PT{y} ) \dd \mathcal H_s ( \PT{x} ) \\ 
& \ge  \frac{1}{ \big( \mathcal H_s (\PT{X})^2 \big)} \left( 1 - \frac{\varepsilon^2}{2} \right)^{2r}    \int_{\PT{X}} \mathcal H_s \big( B_\varepsilon (\PT{x}) \cap \PT{X}  \big)  \dd \mathcal H_s ( \PT{x} ) \\
& \ge \frac{a(\PT{X},s)}{  \mathcal H_s (\PT{X}) }   \left( 1 - \frac{\varepsilon^2}{2} \right)^{2r} \cdot \varepsilon^{s}.
\end{align*}
Assuming that $r > d $ and choosing $\varepsilon = \left( \frac{s}{r} \right)^{1/2}$, we obtain
\begin{equation*}
\left( 1 - \frac{\varepsilon^2}{2} \right)^{2r} \cdot \varepsilon^{s} = \left( 1 - \frac{s}{2r} \right)^{2r}  s^{s/2} r^{-s/2} \ge c_s r^{-s/2},
\end{equation*}
where we use an easy fact that $ \left( 1 - \frac{s}{2r} \right)^{2r}  s^{s/2}  $ is bounded below by a constant depending only on $s$.

Therefore, according to Lemma~\ref{lem:2nd.inner.product.integrals.asymptotics}, arguing as in \eqref{eq:2N} we obtain
\begin{align*}
\int_{\PT{X}} \int_{\PT{X}}  \left( \PT{x} \cdot \PT{y} \right)^{2r}   \dd \mu_{\PT{X}} ( \PT{x} ) \dd \mu_{\PT{X}} ( \PT{y} ) &  - \int_{\mathbb{S}^d} \int_{\mathbb{S}^d} \left( \PT{x} \cdot \PT{y} \right)^{2r}  \dd \sigma_d( \PT{x} ) \dd \sigma_d( \PT{y} )\\
&  \ge c_s  \frac{a(\PT{X},s)}{  \mathcal H_s (\PT{X}) } r^{-s/2} -  c_d r^{-d/2}  \ge c'_{d,s}  \frac{a(\PT{X},s)}{  \mathcal H_s (\PT{X}) } r^{-s/2},
\end{align*}
provided that $r \ge M = c''_{d,s} \left( \frac{  \mathcal H_s (\PT{X}) }{a(\PT{X},s)}\right)^{\frac{2}{d-s}}$. Then \eqref{eq:discrepancy.formula.021} implies the following
\begin{align*}
\left( D_{\mathbb{L}_2,s}^{\mathrm{C}}( \PT{X} ) \right)^2 & \ge C_{d,s} \sum_{r\ge M} \frac{-\Pochhsymb{-\frac{1}{2}}{2r}}{(2r)!} \cdot \frac{a(\PT{X},s)}{  \mathcal H_s (\PT{X}) } r^{-s/2} \\
& \ge C'_{d,s}  \frac{a(\PT{X},s)}{  \mathcal H_s (\PT{X}) }  \sum_{r\ge M} r^{-\frac{3+s}{2}} \ge  C''_{d,s}  \frac{a(\PT{X},s)}{  \mathcal H_s (\PT{X}) }  M^{-\frac{1+s}{2}} \\
& \ge C'''_{d,s} \left( \frac{  \mathcal H_s (\PT{X}) }{a(\PT{X},s)}\right)^{- \frac{1+s}{d-s} -1 } = C'''_{d,s} \left( \frac{  \mathcal H_s (\PT{X}) }{a(\PT{X},s)}\right)^{- \frac{d+1}{d-s}  },
\end{align*}
which finishes the proof of the theorem. 
\end{proof}

Since the case of  curves ($s=1$) presents particular interest due to \cite{EhGr}, we state it as a separate corollary. Observe that whenever $\Gamma \subset \mathbb S^d$ is a rectifiable curve with $\ell (\Gamma) = H_1 (\Gamma) \ge 1$, then automatically $a (\Gamma, 1 ) \ge \frac12$. Therefore we immediately obtain the following bound for the discrepancy of curves.

\begin{cor}
There exists a constant $c_{d,1}$, such that for any rectifiable curve $\Gamma \subset \mathbb S^d$  with  length $\ell (\Gamma) \ge 1$, its  discrepancy {satisfies}
\begin{equation}
D_{\mathbb{L}_2,1}^{\mathrm{C}}( \Gamma ) \ge c_{d,1} \left( \ell(\Gamma) \right)^{- \frac{d+1}{2(d-1)}}.
\end{equation}
\end{cor}

\section{Refinements: Analysis of the constant $c_d$ in \eqref{eq:Beck}}
\label{sec:constants}

A more careful analysis of our proof reveals information about the implicit constant~$c_d$ in inequality \eqref{eq:Beck}. Several steps in the argument can be somewhat sharpened and improved. 
We obtain uniform bounds and consider asymptotic bounds. A comparison with conjectured bounds shows that our bounds are surprisingly tight in the sense that by only using diagonal terms in the point sums, we recover one significant digit in the conjectured asymptotic constant and the relative error is small; cf. Table~\ref{tbl:summary}. 
%

%

%


\subsection{Uniform bound}
The coefficients $\frac{\Pochhsymb{\frac{1}{2}}{r}}{\Pochhsymb{\frac{d+1}{2}}{r}}$ in \eqref{eq:discrepancy.formula.lower.bound.01} are strictly monotonically decreasing as can be seen from
\begin{equation} \label{eq:monotonicity.ip.2r.int}
\frac{\Pochhsymb{\frac{1}{2}}{r+1}}{\Pochhsymb{\frac{d+1}{2}}{r+1}} 
=
\frac{\Pochhsymb{\frac{1}{2}}{r}}{\Pochhsymb{\frac{d+1}{2}}{r}} \, \frac{r + \frac{1}{2}}{r + \frac{d+1}{2}}
< \frac{\Pochhsymb{\frac{1}{2}}{r}}{\Pochhsymb{\frac{d+1}{2}}{r}}, \qquad r \in \mathbb{N}_0.
\end{equation}
Thus, we may estimate (without asymptotic reasoning but still keeping only the diagonal terms in the double sum and therefore no further assumptions on the geometric properties of the points) 
\begin{equation} \label{eq:parenthetical.non.negativity}
\frac{1}{N^2} \sum_{j=1}^N \sum_{k=1}^N \left( \PT{x}_j \cdot \PT{x}_k \right)^{2r} - \frac{\Pochhsymb{\frac{1}{2}}{r}}{\Pochhsymb{\frac{d+1}{2}}{r}} 
\geq 
\frac{1}{N} - \frac{\Pochhsymb{\frac{1}{2}}{M}}{\Pochhsymb{\frac{d+1}{2}}{M}}, \qquad r \geq M.
\end{equation}
For $M$ sufficiently large, the right-hand side of above inequality is positive and thus 
\begin{equation*}
\left( D_{\mathbb{L}_2}^{\mathrm{C}}( \PT{x}_1, \dots, \PT{x}_N ) \right)^2 
\geq
\sqrt{2} \, C_d \left( \frac{1}{N} - \frac{\Pochhsymb{\frac{1}{2}}{M}}{\Pochhsymb{\frac{d+1}{2}}{M}} \right) \sum_{r=M}^\infty \frac{-\Pochhsymb{-\frac{1}{2}}{2r}}{(2r)!}.
\end{equation*}
Using the upper bound in Gautschi’s Inequality for ratios of gamma functions \cite[Eq.~5.6.4]{NIST:DLMF}, we may estimate 
\begin{equation} \label{eq:a.m.lower.bound}
\frac{-\Pochhsymb{-\frac{1}{2}}{m}}{m!} = \frac{-1}{\gammafcn( - \frac{1}{2} )} \frac{\gammafcn( m - \frac{1}{2} )}{\gammafcn( m + 1 )} = \frac{1}{2 \sqrt{\pi}} \, \frac{1}{m} \, \frac{\gammafcn( m - 1 + \frac{1}{2} )}{\gammafcn( m - 1 + 1 )} > \frac{1}{2 \sqrt{\pi}} \, \frac{1}{m} \, \frac{1}{m^{\frac{1}{2}}} = \frac{1}{2 \sqrt{\pi}} \, \frac{1}{m^{\frac{3}{2}}}
\end{equation}
for all $m \in \mathbb{N}$ and get
\begin{equation*}
\sum_{r=M}^\infty \frac{-\Pochhsymb{-\frac{1}{2}}{2r}}{(2r)!} > \frac{1}{4 \sqrt{2\pi}} \sum_{r=M}^\infty \frac{1}{r^{\frac{3}{2}}} > \frac{1}{4 \sqrt{2\pi}} \int_M^\infty \frac{1}{x^{\frac{3}{2}}} \dd x = \frac{1}{2 \sqrt{2} \, \sqrt{\pi}} \, \frac{1}{M^{\frac{1}{2}}}, \qquad M \geq 1.
\end{equation*}
A more tedious computation shows that the lower bound is the correct leading term in the asymptotic expansion of the sum on the left-hand side; cf. Lemma~\ref{lem:sum.asymptotics.alpha} for $d = 0$ and $\alpha = 1$. 

By Lemma~\ref{lem:2nd.inner.product.integrals.asymptotics}, we have
\begin{equation*}
\left( D_{\mathbb{L}_2}^{\mathrm{C}}( \PT{x}_1, \dots, \PT{x}_N ) \right)^2 
\geq
\frac{C_d}{2 \sqrt{\pi}} \left( \frac{1}{N} - \frac{\Pochhsymb{\frac{1}{2}}{M}}{\Pochhsymb{\frac{d+1}{2}}{M}} \right) \frac{1}{M^{\frac{1}{2}}} 
> 
\frac{C_d}{2 \sqrt{\pi}} \left( \frac{1}{N} - \frac{\gammafcn( \frac{d+1}{2} )}{\gammafcn( \frac{1}{2} )} \; \frac{1}{M^{\frac{d}{2}}} \right) \frac{1}{M^{\frac{1}{2}}}
\end{equation*}
provided $M$ is sufficiently large to guarantee positivity of the lower bound. 
Substituting $M = (c N)^{\frac{2}{d}}$, we arrive at the lower bound
\begin{equation*}
\left( D_{\mathbb{L}_2}^{\mathrm{C}}( \PT{x}_1, \dots, \PT{x}_N ) \right)^2 
\geq
\frac{C_d}{2 \sqrt{\pi}}  \, \frac{F(c)}{N^{1+\frac{1}{d}}}, \qquad F(c) \DEF \left( c - \frac{\gammafcn(\frac{d+1}{2})}{\gammafcn(\frac{1}{2})} \right) \frac{1}{c^{1+\frac{1}{d}}}. 
\end{equation*}
The parameter $c$ has to be larger than $\frac{\gammafcn(\frac{d+1}{2})}{\gammafcn(\frac{1}{2})}$ for a non-trivial lower bound. 
The function $c \stackrel{F}{\mapsto} ( c - \beta ) \frac{1}{c^{1+\gamma}}$ with fixed $\beta, \gamma > 0$ has in $[\beta, \infty)$ a unique maximum at $c^* = \beta ( 1 + \frac{1}{\gamma} ) > \beta$ with value $F(c^*) = \frac{1}{( 1 + \gamma ) \big( \beta ( 1 + \frac{1}{\gamma} ) \big)^{\gamma}}$. 
%
%
Thus, for all $N \geq 2$, 
\begin{equation} \label{eq:uniform.lower.bound}
D_{\mathbb{L}_2}^{\mathrm{C}}( \PT{x}_1, \dots, \PT{x}_N ) 
\geq
c_d^* \, N^{-\frac{1}{2}-\frac{1}{2d}}, \qquad c_d^* = \sqrt{ \frac{C_d}{2 \sqrt{\pi}} \frac{1}{1+\frac{1}{d}} \left( \frac{\gammafcn( \frac{3}{2} )}{\gammafcn( \frac{d+3}{2} )} \right)^{\frac{1}{d}} }.
\end{equation}
In particular, we get 
\begin{equation*}
c_2^* = \frac{1}{6 \sqrt{6 \pi }} = 0.1959291678902056\dots.
\end{equation*}

\subsection{Asymptotic constant}
\label{subsec:asymptotic.constant}

We derive a lower bound for 
\begin{equation*}
\liminf_{\substack{N\to\infty \\ \PT{x}_1, \dots, \PT{x}_N \in \mathbb{S}^d }} N^{\frac{1}{2}+\frac{1}{2d}} D_{\mathbb{L}_2}^{\mathrm{C}}( \PT{x}_1, \dots, \PT{x}_N )
\end{equation*}
that agrees with one significant digit (and recovers more than $93\%$) of the conjectured constant of the leading term of the minimal $\mathbb{L}_2$-discrepancy asymptotics available for dimensions $2$, $4$, $8$, and $24$. 
%
Surprisingly, this result uses only the diagonal terms of the double sum for the points in \eqref{eq:discrepancy.formula.02} and discards terms with $m < 2M$ and $M$ of order $N^{\frac{2}{d}}$. 

We rewrite the odd terms in \eqref{eq:discrepancy.formula.02} as
\begin{equation*}
\frac{1}{N^2} \sum_{j,k=1}^N ( \PT{x}_j \cdot \PT{x}_k )^{2r+1} 
=
\frac{1}{N^2} \sum_{j,k=1}^N ( \PT{x}_j \cdot \PT{x}_k )^{2r} \left( 1 + \PT{x}_j \cdot \PT{x}_k \right) - \frac{1}{N^2} \sum_{j,k=1}^N ( \PT{x}_j \cdot \PT{x}_k )^{2r}.
\end{equation*}
For convenience, set $a_m \DEF \frac{-\Pochhsymb{-\frac{1}{2}}{m}}{m!}$. 
Then, considering both even and odd terms with ${m \geq 2M}$,
\begin{align*}
&\frac{\left( D_{\mathbb{L}_2}^{\mathrm{C}}( \PT{x}_1, \dots, \PT{x}_N ) \right)^2}{\sqrt{2} \, C_d} 
\geq
\sum_{r=M}^\infty a_{2r} \left( \frac{1}{N^2} \sum_{j,k=1}^N ( \PT{x}_j \cdot \PT{x}_k )^{2r} - \frac{\Pochhsymb{\frac{1}{2}}{r}}{\Pochhsymb{\frac{d+1}{2}}{r}} \right) + \sum_{r=M}^\infty a_{2r+1} \frac{1}{N^2} \sum_{j,k=1}^N ( \PT{x}_j \cdot \PT{x}_k )^{2r+1} \\
&\phantom{equ}=
\sum_{r=M}^\infty a_{2r+1} \frac{1}{N^2} \sum_{j,k=1}^N ( \PT{x}_j \cdot \PT{x}_k )^{2r} \left( 1 + \PT{x}_j \cdot \PT{x}_k \right) - \sum_{r=M}^\infty a_{2r} \, \frac{\Pochhsymb{\frac{1}{2}}{r}}{\Pochhsymb{\frac{d+1}{2}}{r}}
+ \sum_{r=M}^\infty \left( a_{2r} - a_{2r+1} \right) \frac{1}{N^2} \sum_{j,k=1}^N ( \PT{x}_j \cdot \PT{x}_k )^{2r}.
\end{align*}
By strict monotonicity (see \eqref{eq:monotonicity.ip.2r.int}), $a_{2r} - a_{2r+1} > 0$ for all $r \geq 1$. Thus,  using only the diagonal terms in the point sums, we get 
\begin{equation*}
\left( D_{\mathbb{L}_2}^{\mathrm{C}}( \PT{x}_1, \dots, \PT{x}_N ) \right)^2 
\geq 
\sqrt{2} \, C_d \left( \frac{2}{N} \sum_{r=M}^\infty a_{2r+1} - \sum_{r=M}^\infty a_{2r} \, \frac{\Pochhsymb{\frac{1}{2}}{r}}{\Pochhsymb{\frac{d+1}{2}}{r}} + \frac{1}{N} \sum_{r=M}^\infty \left( a_{2r} - a_{2r+1} \right) \right).
\end{equation*}
Reordering terms, we arrive at
\begin{equation} \label{eq:pre.bound.asymp.const}
\left( D_{\mathbb{L}_2}^{\mathrm{C}}( \PT{x}_1, \dots, \PT{x}_N ) \right)^2 
\geq 
\sqrt{2} \, C_d \left( \frac{1}{N} \sum_{m=2M}^\infty \frac{-\Pochhsymb{-\frac{1}{2}}{m}}{m!} - \sum_{r=M}^\infty \frac{-\Pochhsymb{-\frac{1}{2}}{2r}}{(2r)!} \, \frac{\Pochhsymb{\frac{1}{2}}{r}}{\Pochhsymb{\frac{d+1}{2}}{r}} \right).
\end{equation}
We proved that it suffices to consider the diagonal terms in the double sum for the points (even in the case of odd $m$) for all $m \geq 2M$ for any $M \geq 1$. 

We need the following asymptotic results proven in Appendix~\ref{app:proof.lem:sum.asymptotics.alpha}. 
\begin{lem} \label{lem:sum.asymptotics.alpha}
Let $0 < \alpha < 2$. Then
\begin{equation*}
\sum_{m=2M}^\infty \frac{-\Pochhsymb{-\frac{\alpha}{2}}{m}}{m!} 
= 
\frac{2^{-\frac{\alpha}{2}}}{\gammafcn( 1 - \frac{\alpha}{2} )} \, \frac{1}{M^{\frac{\alpha}{2}}} + \mathcal{O}\Big( \frac{1}{M^{1+\frac{\alpha}{2}}} \Big) \qquad \text{as $M \to \infty$;}
\end{equation*}
and for $d = 0$ and $d \geq 1$, 
\begin{align*}
\sum_{r=M}^\infty \frac{-\Pochhsymb{-\frac{\alpha}{2}}{2r}}{(2r)!} \, \frac{\Pochhsymb{\frac{1}{2}}{r}}{\Pochhsymb{\frac{d+1}{2}}{r}}
=
\frac{1}{M^{\frac{d}{2}}} \left( 2^{-\frac{\alpha}{2}} \frac{\alpha \gammafcn( \frac{d+1}{2} )}{2 \sqrt{\pi} \; \gammafcn( 1 - \frac{\alpha}{2} ) ( d + \alpha )} \, \frac{1}{M^{\frac{\alpha}{2}}} + \mathcal{O}\Big( \frac{1}{M^{1+\frac{\alpha}{2}}} \Big) \right) \quad \text{as $M\to \infty$.}
\end{align*}
\end{lem}

By Lemma~\ref{lem:sum.asymptotics.alpha} for $\alpha = 1$, 
\begin{equation*}
N^{1+\frac{1}{d}} \left( D_{\mathbb{L}_2}^{\mathrm{C}}( \PT{x}_1, \dots, \PT{x}_N ) \right)^2
\geq
\frac{C_d}{\sqrt{\pi}} \left( 1 -  \frac{\gammafcn( \frac{d+1}{2} )}{2 \sqrt{\pi} \, ( d + 1 )} \, \frac{N}{M^{\frac{d}{2}}} \right) \frac{N^{\frac{1}{d}}}{M^{\frac{1}{2}}} + \mathcal{O}\Big( \frac{N^{\frac{1}{d}}}{M^{\frac{3}{2}}} \Big) + \mathcal{O}\Big( \frac{N^{1+\frac{1}{d}}}{M^{\frac{d+3}{2}}} \Big).
\end{equation*}
On assuming that $M = (c N)^{\frac{2}{d}}$, $c > \frac{\gammafcn( \frac{d+1}{2} )}{2 \sqrt{\pi} \, ( d + 1 )}$, we get 
\begin{equation*}
N^{1+\frac{1}{d}} \left( D_{\mathbb{L}_2}^{\mathrm{C}}( \PT{x}_1, \dots, \PT{x}_N ) \right)^2 
\geq 
\frac{C_d}{\sqrt{\pi}} \left( c - \frac{\gammafcn( \frac{d+1}{2} )}{2 \sqrt{\pi} \, ( d + 1 )} \right) \frac{1}{c^{1+\frac{1}{d}}} + \mathcal{O}\Big( \frac{1}{N^{\frac{2}{d}}} \Big).
\end{equation*}
Thus,
\begin{equation*}
\liminf_{\substack{N\to\infty \\ \PT{x}_1, \dots, \PT{x}_N \in \mathbb{S}^d }} N^{1+\frac{1}{d}} \left( D_{\mathbb{L}_2}^{\mathrm{C}}( \PT{x}_1, \dots, \PT{x}_N ) \right)^2
\geq 
\frac{C_d}{\sqrt{\pi}} \left( c - \frac{\gammafcn( \frac{d+1}{2} )}{2 \sqrt{\pi} \, ( d + 1 )} \right) \frac{1}{c^{1+\frac{1}{d}}}.
\end{equation*}
The function $c \stackrel{F}{\mapsto} ( c - \beta ) \frac{1}{c^{1+\gamma}}$ with fixed $\beta, \gamma > 0$ has in $[\beta, \infty)$ a unique maximum at $c^* = \beta ( 1 + \frac{1}{\gamma} )$ with value $F(c^*) = \frac{1}{( 1 + \gamma ) \big( \beta ( 1 + \frac{1}{\gamma} ) \big)^{\gamma}}$. 
(We remark that the optimal $M^* = (c^*N)^{\frac{1}{d}}$ is smaller then the minimal $M$ for which $\frac{1}{N} - \frac{\Pochhsymb{\frac{1}{2}}{M}}{\Pochhsymb{\frac{d+1}{2}}{M}} \geq 0$; i.e., inclusion of the odd terms and rearrangement of terms extends the range of $m$ below the value of $M$ for which $\frac{1}{N} - \frac{\Pochhsymb{\frac{1}{2}}{M}}{\Pochhsymb{\frac{d+1}{2}}{M}}$ is a non-negative and thus useful bound in \eqref{eq:parenthetical.non.negativity}.)
Thus,
\begin{equation*}
\liminf_{\substack{N\to\infty \\ \PT{x}_1, \dots, \PT{x}_N \in \mathbb{S}^d }} N^{\frac{1}{2}+\frac{1}{2d}} D_{\mathbb{L}_2}^{\mathrm{C}}( \PT{x}_1, \dots, \PT{x}_N ) 
\geq 
c_d^{***},
\end{equation*}
where
\begin{equation} \label{eq:c.d.***}
c_d^{***} = \sqrt{\frac{C_d}{\sqrt{\pi}} F(c^*)} = \sqrt{\frac{C_d}{\sqrt{\pi}} \frac{1}{1+\frac{1}{d}} \left( \frac{2\gammafcn(\frac{1}{2})}{\gammafcn(\frac{d+1}{2})} \right)^{\frac{1}{d}}}.
\end{equation}
In particular, we get 
\begin{equation} \label{eq:c.2.***}
c_2^{***} = \frac{1}{\sqrt{3 \sqrt{\pi }}} = 0.4336625352920387\dots.
\end{equation}


\subsection{Comparison with conjectured asymptotic behaviour} 
\label{subsec:comparison}

Maximal sum of distance points $\PT{x}_1^*, \dots, \PT{x}_N^* \in \mathbb{S}^d$ have according to Stolarsky's invariance principle \eqref{eq:Stolarsky.s.invariance.principle} the smallest spherical cap $\mathbb{L}_2$-discrepancy. Utilizing a fundamental conjecture for the Riesz $s$-energy (cf. \cite{BrHaSa2012b,BoHaSaBook2019}) the second author suggested in \cite{Br2011} that the following limit exists
\begin{equation} \label{eq:minimal.L2.discrepancy.conjecture}
\lim_{N\to\infty} N^{\frac{1}{2}+\frac{1}{2d}} D_{\mathbb{L}_2}^{\mathrm{C}}( \PT{x}_1^*, \dots, \PT{x}_N^* )  
=
c_d^{\mathrm{conj}}
\DEF
\sqrt{- C_d \left[ \frac{C_{s,d}}{( \mathcal{H}_d( \mathbb{S}^d ) )^{\frac{s}{d}}} \right]_{s=-1}}.
\end{equation}
The limiting value 
depends on $C_d$ given in \eqref{eq:Cd} and the constant
\begin{equation*}
\frac{C_{s,d}}{( \mathcal{H}_d( \mathbb{S}^d ) )^{\frac{s}{d}}}
\end{equation*}
appearing in the large $N$ asymptotic expansion of the Riesz $s$-energy as the conjectured constant of the second term if $-2 < s < d + 2$ for $s \neq 0, d$ and proven to be the  constant of the leading term for $s > d$. 
It is composed of the manifold depending $d$-dimensional Hausdorff measure (normalized so that the unit cube has measure one) of $\mathbb{S}^d$, 
\begin{equation*}
\mathcal{H}_d( \mathbb{S}^d ) = \omega_d \DEF \frac{2 \; \pi^{\frac{d+1}{2}}}{\gammafcn(\frac{d+1}{2})},
\end{equation*}
and the manifold independent constant $C_{s,d}$. 
For $s > d$, the constant $C_{s,d}$ can be obtained as the limit of the properly normalized minimal Riesz $s$-energy $\mathcal{E}_s( A; N )$ for $N$ points taken, say, from a $d$-dimensional unit cube $A=[0,1]^d$:\footnote{In fact, this constant does not depend on the set $A$ for a large class of sets; cf. \emph{Poppy-seed Bagel Theorem}~\cite{BoHaSaBook2019}.} 
\begin{equation} \label{eq:def.C.s.d}
C_{s,d} \DEF \lim_{N \to \infty} \frac{\mathcal{E}_s( [0,1]^d; N )}{N^{1+\frac{s}{d}}}.
\end{equation}
Based on work on universal optimality of the integer lattice $\mathbb{Z}$, the $\mathbf{E}_8$-lattice, and the Leech lattice $\boldsymbol{\Lambda}_{24}$ in \cite{CoKu2007} and the celebrated paper \cite{CoKuMiRaVi2022}, the constant is known in these dimensions to be a multiple of the Epstein zeta function of the lattice: 
\begin{equation} \label{eq:C.s.d.for.d.EQ.1.8.24}
C_{s,d} = | \Lambda |^{\frac{s}{d}} \zetafcn_{\Lambda}\big( \frac{s}{2} \big) = | \Lambda |^{\frac{s}{d}} \sum_{\PT{0} \neq \PT{x} \in \Lambda} \frac{1}{( \PT{x} \cdot \PT{x} )^{\frac{s}{2}}}, \qquad s > d,
\end{equation}
where $\Lambda = \mathbb{Z}$ ($d=1$), $\Lambda = \mathbf{E}_8$ ($d=8$), and $\Lambda = \boldsymbol{\Lambda}_{24}$  ($d=24$). The scaling factor is a power of the co-volume $| \Lambda |$ of the lattice. 
It is conjectured that \eqref{eq:C.s.d.for.d.EQ.1.8.24} also holds for the hexagonal lattice $\boldsymbol{A}_2$ ($d=2$) and the checkerboard lattice $\boldsymbol{D}_4$ ($d = 4$). (For other values of~$d$ it is not known which form a conjecture for the value of $C_{s,d}$ should take. 
%
For example, for the purpose of $\mathbb{L}_2$-discrepancy, a universality property may only need to hold for power laws $\| \PT{x} - \PT{y} \|^\alpha$ for $0 < \alpha < 2$.
However, it is strongly believed that for every dimension $d$, the constant $C_{s,d}$ obeys a `principle of analytic continuation' in the sense that the coefficient of the second term of the Riesz energy asymptotics for $-2 < s < d + 2$ and $s \neq 0, d$ is given by the analytic continuation in the complex $s$-plane of $C_{s,d}$ defined in \eqref{eq:def.C.s.d}.) 

The Epstein zeta function $\zetafcn_{\Lambda}(s)$ of the full-rank lattice $\Lambda \subset \mathbb{R}^d$ defined for $s > d$ by 
\begin{equation*}
\zetafcn_\Lambda( s ) \DEF \sum_{\PT{0} \neq \PT{x} \in \Lambda} \frac{1}{( \PT{x} \cdot \PT{x} )^s}
\end{equation*}
has an analytic continuation to $\mathbb{C}$ except for a simple pole at $s = \frac{d}{2}$ with residue
$
\frac{d}{2} \, \frac{\mathcal{H}_d( \mathbb{B}^d )}{| \Lambda |} = \frac{\pi^{\frac{d}{2}}}{\gammafcn( \frac{d}{2} )} \, \frac{1}{| \Lambda |}
$
.
%
Thus, when \eqref{eq:C.s.d.for.d.EQ.1.8.24} holds, the constant $C_{s,d}$ has a simple pole at $s = d$ with residue 
\begin{equation*}
\lim_{s \to d} \left( s - d \right) C_{s,d} = 2 \, \frac{\pi^{\frac{d}{2}}}{\gammafcn( \frac{d}{2} )}
\end{equation*}
and is analytic elsewhere in the complex $s$-plane. 
Hence, when \eqref{eq:C.s.d.for.d.EQ.1.8.24} holds, we get 
\begin{equation} \label{eq:c.d.conj.via.lattices}
c_d^{\mathrm{conj}} 
= 
\sqrt{C_d \left( \frac{\mathcal{H}_d( \mathbb{S}^d )}{| \Lambda |} \right)^{\frac{1}{d}} \left( - \zetafcn_{\Lambda}\big( -\frac{1}{2} \big) \right)} \\
=
\sqrt{C_d \left( \omega_d \right)^{\frac{1}{d}} \left( - | \Lambda |^{-\frac{1}{d}} \zetafcn_{\Lambda}\big( -\frac{1}{2} \big) \right)}, 
\end{equation}
where $d$ is the dimension of the lattice $\lambda \in \{ \boldsymbol{A}_2, \boldsymbol{D}_4, \boldsymbol{E}_8, \boldsymbol{\Lambda}_{24} \}$.

We compare now our asympotic constant in \eqref{eq:c.d.***}, which can be written as 
\begin{equation*}
c_d^{***} = \sqrt{C_d \left( \omega_d \right)^{\frac{1}{d}} \frac{1}{\pi} \, \frac{1}{1+\frac{1}{d}}},
\end{equation*}
with the conjectured one in \eqref{eq:c.d.conj.via.lattices}.
Mathematica is used to compute the expressions and their values to $18$ digits accuracy. The numerical values are truncated (not rounded). The relative error is rounded up. 
Table~\ref{tbl:summary} provides a summary of our findings. In all cases there is agreement in one significant digit with $c_d^{\mathrm{conj}} > c_d^{***}$ and the relative error grows from about $3\%$ (if $d = 2$) to about $7\%$ (if $d = 24$). This gradual increase in the relative error for increasing dimension $d$ suggests that our lower bound does not completely capture the dependence on the dimension; cf. Figure~\ref{fig:c3star}.
%
\begin{table}[ht]
\caption{\label{tbl:summary} Comparison of conjectured $c_d^{\mathrm{conj}}$ and asymptotic $c_d^{***}$ bound for the cases when \eqref{eq:C.s.d.for.d.EQ.1.8.24} holds or is conjectured to hold.}
\begin{tabular}{r|ll|ll|l}
$d$ & $c_d^{\mathrm{conj}}$ & $c_d^{***}$ & $c_d^{\mathrm{conj}} - c_d^{***}$ & $\frac{c_d^{\mathrm{conj}} - c_d^{***}}{c_d^{\mathrm{conj}}}$ & Remarks \\
\hline
 $2$ & $0.4467972835040832\dots$ & $0.4336625352920387\dots$ & $0.013\dots$ & $3\%$ & $\boldsymbol{A}_2$; conj. univ. \\
 $4$ & $0.3426606934243682\dots$ & $0.3288548512164972\dots$ & $0.013\dots$ & $4\%$ & $\boldsymbol{D}_4$; non-univ. \\
 $8$ & $0.2558385395385698\dots$ & $0.2431072174822736\dots$ & $0.012\dots$ & $5\%$ & $\boldsymbol{E}_8$; univ. \\
$24$ & $0.1557897704986152\dots$ & $0.1451892677457039\dots$ & $0.010\dots$ & $7\%$ & $\boldsymbol{\Lambda}_{24}$; univ. \\
\end{tabular}
\end{table}

\begin{figure}[ht]
\centering
\includegraphics[scale=.2]{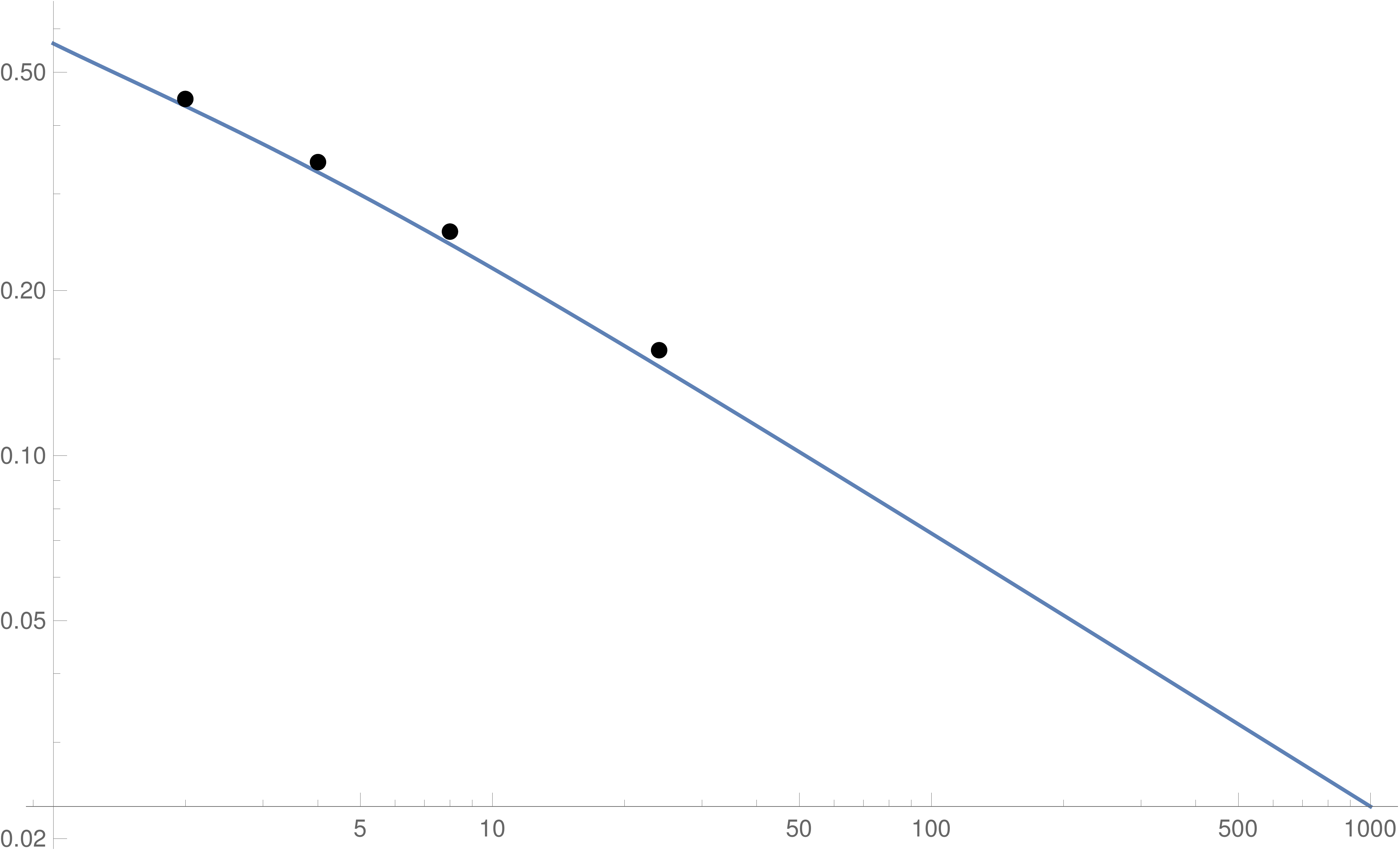} \\
\caption{\label{fig:c3star} The behaviour of $c_d^{***}$ as a function of (continuous) $d$. Points indicate the values of $c_d^{\mathrm{conj}}$ for $d = 2, 4, 8, 16$.}
\end{figure}

\subsubsection{The case $d=2$} The hexagonal or equi-triangular lattice $\boldsymbol{A}_2$  generated by the vectors $(1,0)$ and $(\frac{1}{2}, \frac{\sqrt{3}}{2})$ is conjectured to be universally optimal. It has co-volume $| \boldsymbol{A}_2 | = \frac{\sqrt{3}}{2}$ and has the Epstein zeta function 
\begin{equation*}
\zetafcn_{\boldsymbol{A}_2}( s ) = 6 \zetafcn( s ) \, 3^{-s} \left( \zetafcn\big(s,\frac{1}{3}\big) - \zetafcn\big(s,\frac{2}{3}\big) \right), \qquad s > 1,
\end{equation*}
given in terms of the Riemann zeta function and the Hurwitz zeta function 
\begin{equation*}
\zetafcn( s, a ) \DEF \sum_{k=0}^\infty \frac{1}{( k + a )^s}, \qquad \re s > 1.
\end{equation*}
We get 
\begin{equation*}
c_2^{\mathrm{conj}} = 0.4467972835040832\ldots.
\end{equation*}
Compared with \eqref{eq:c.2.***}, we observe that our bound agrees in one significant digit; more information in Table~\ref{tbl:summary}.

\subsubsection{The case $d=4$} The checkerboard lattice $\boldsymbol{D}_4$ is given by the vectors in $\mathbb{Z}^4$ with coordinates that sum to an even number. The associated $\boldsymbol{D}_4$ root system (i.e., the set of vertices of the regular 24-cell) is known to be not universally optimal; see~\cite{CoCoElKu2007}. When rescaled so that the co-volume is $1$, the Epstein zeta function is given by (cf.  \cite{SaSt2006}): 
\begin{equation*} 
\zetafcn_{\boldsymbol{D}_4}( s ) = 24 \; 2^{-\frac{s}{2}} \left( 1 - 2^{1-s} \right) \zetafcn( s ) \; \zetafcn( s - 1 ), \qquad \re s > 2.
\end{equation*}
We get
\begin{equation*}
c_4^{\mathrm{conj}} = 0.3426606934243682\dots.
\end{equation*}
Our asymptotic lower bound, given by \eqref{eq:c.d.***}, is
\begin{equation*}
c_4^{***} = \frac{3^{\frac{3}{8}}}{2^{\frac{5}{8}} \sqrt{5} \; \pi^{\frac{1}{4}}} = 0.3288548512164972\dots.
\end{equation*}
Our bound agrees in one significant digit; cf. Table~\ref{tbl:summary}.

\subsubsection{The case $d=8$ and $d = 24$} Both the $\boldsymbol{E}_8$ lattice and the Leech lattice $\boldsymbol{\Lambda}_{24}$ are universally optimal as shown in the celebrated paper \cite{CoKuMiRaVi2022}. Either lattice is unimodular (i.e. has co-volume $1$). The $\boldsymbol{E}_8$ lattice has a root system. The Leech lattice has no root system. Further properties can be found in \cite{CoSl1988}. Their Epstein zeta functions are (cf. \cite{SaSt2006}):
\begin{align*}
\zetafcn_{\boldsymbol{E}_8}( s ) &= 240 \; 2^{-s} \zetafcn( s ) \zetafcn( s - 3 )
\intertext{and}
\zetafcn_{\boldsymbol{\Lambda}_{24}}( s ) &= \frac{65520}{691} \; 2^{-s} \left( \zetafcn( s ) \zetafcn( s - 11 ) - \RamanujanL(s, \Delta) \right), 
\end{align*}
where 
\begin{equation*}
\RamanujanL(s, \Delta) \DEF \sum_{k=1}^\infty \frac{\tau( k )}{k^s}
\end{equation*}
denotes the Ramanujan tau Dirichlet L-function. 

Using Mathematica, we get 
\begin{equation*}
c_8^{\mathrm{conj}} = 0.2558385395385698\dots
\end{equation*}
which we compare with 
\begin{equation*}
c_8^{***} = \frac{35^{\frac{7}{16}}}{12\ 2^{\frac{3}{16}} 3^{\frac{1}{16}} \pi^{\frac{1}{4}}} = 0.2431072174822736\dots,
\end{equation*}
whereas
\begin{equation*}
c_{24}^{\mathrm{conj}} = 0.1557897704986152\dots
\end{equation*}
which we compare with 
\begin{equation*}
c_{24}^{***} = \frac{3^{\frac{19}{48}} 7^{\frac{11}{24}} 96577^{\frac{23}{48}}}{2560\ 2^{\frac{35}{48}} 5^{\frac{1}{24}} 11^{\frac{1}{48}}  \pi^{\frac{1}{4}}} = 0.1451892677457039\dots.
\end{equation*}
We observe agreement in one significant digit; cf. Table~\ref{tbl:summary}.

\subsection{Further remarks} 


Contrary to the usual approach in proving lower bounds, we consider the series remainder in  \eqref{eq:discrepancy.formula.02} with $m \geq 2M$ for some suitable $M$. This strategy is aided, in part, by the fact that for a spherical $t$-design all terms with $m \leq t$ vanish because of exactness of integration. 
By the celebrated result \cite{BoRaVi2013}, for every $N \geq \gamma_{d} \, t^d$, there exists a spherical $t$-design with $N$ points on $\mathbb{S}^d$ for which equality holds in \eqref{eq:pos} for all $1 \leq m \leq t$. The constant $\gamma_{d}$ depends on $d$ only. Thus, for a spherical $t$-design with $t = \frac{1}{(\gamma_d)^{\frac{1}{d}}} N^{\frac{1}{d}}$, all terms with $1 \leq m \leq \frac{1}{(\gamma_d)^{\frac{1}{d}}} N^{\frac{1}{d}}$ in \eqref{eq:discrepancy.formula.02} vanish. 
Heuristically, well-distributed points seem to approximate the behaviour of spherical designs in the sense that 
\begin{equation*}
S(m) \DEF \frac{1}{N^2} \sum_{j=1}^N \sum_{k=1}^N \left( \PT{x}_j \cdot \PT{x}_k \right)^m - \int_{\mathbb{S}^d} \int_{\mathbb{S}^d} \left( \PT{x} \cdot \PT{y} \right)^m  \dd \sigma_d( \PT{x} ) \dd \sigma_d( \PT{y} ), \qquad m \geq 1,
\end{equation*}
goes to zero qualitatively faster than $\frac{1}{N}$ if $1 \leq m \leq c_\eps N^{\frac{1}{d}-\eps}$, $c_\eps > 0$ and $\eps > 0$.
Indeed, the matching upper bound in \eqref{eq:Beck} valid for maximal sum of distance points or QMC-design sequences of sufficient strength $s \geq \frac{d+1}{2}$ (cf. \cite{BrSaSlWo2014}) implies that individual terms in the expansion \eqref{eq:discrepancy.formula.02} satisfy 
\begin{equation*}
N S(m) \leq \frac{C}{N^{\frac{1}{d}}} \, \frac{m!}{-\Pochhsymb{-\frac{1}{2}}{m}} \leq \widetilde{C} \, \frac{m^{\frac{3}{2}}}{N^{\frac{1}{d}}} \leq \widetilde{C} \, \frac{M^{\frac{3}{2}}}{N^{\frac{1}{d}}}, \qquad m = 1, 2, 3, \dots, M.
\end{equation*}
The constants $C$ and $\widetilde{C}$ depend on $d$ and the sequence of point sets and the coefficient in the inequality \eqref{eq:a.m.lower.bound}. The right-hand side goes to zero like $\frac{1}{N^{\frac{1}{d}-\frac{3}{2}\alpha}}$ if $M \sim N^\alpha$ and $0 \leq \alpha < \frac{2}{3} \frac{1}{d}$. 
A simple summation by parts argument utilizing strict monotonicity of $\frac{-\Pochhsymb{-\frac{1}{2}}{m}}{m!}$ gives 
\begin{equation*}
\frac{1}{M} \sum_{m=1}^M (N S(m)) = \mathcal{O}\Big( \frac{M^{\frac{1}{2}}}{N^{\frac{1}{d}}} \Big) = \mathcal{O}\Big( \frac{1}{N^{\frac{1}{d}-\frac{1}{2}\alpha}} \Big) \qquad \text{as $N \to \infty$,}
\end{equation*}
provided $M \sim N^\alpha$ with $0\leq \alpha < \frac{2}{d}$. 
This phenomenon is also evident in the decay of so-called structure factors for hyperuniform sequences of point sets (if $m$ stays bounded; cf.~\cite{BrGrKu2019}). 
The precise decay rate of $N S_m$ (or the structure factor of hyperuniformity) as $N \to \infty$ for $m$ stays bounded,  grows weakly to infinity, or grows like $N^{\frac{1}{d}}$ (threshold growth) is an ongoing research quest for (constructible) sequence of point sets on $\mathbb{S}^d$.

On the other hand, the series expansion \eqref{eq:discrepancy.formula.02} can be truncated after a sufficiently large~$M$. Because $S(m)$ is bounded by $1$ from above, a coarse estimate shows that 
\begin{equation*}
\sum_{r=2M}^\infty \frac{-\Pochhsymb{-\frac{1}{2}}{m}}{m!} S(m) 
=
\mathcal{O}\Big( \frac{1}{M^{\frac{1}{2}}} \Big) 
=
\mathcal{O}\Big( \frac{N^{1+\frac{1}{d}}}{M^{\frac{1}{2}}} \, \frac{1}{N^{1+\frac{1}{d}}} \Big) 
=
o\Big( \frac{1}{N^{1+\frac{1}{d}}} \Big)
\end{equation*}
as $M, N \to \infty$, provided that $\frac{N^{1+\frac{1}{d}}}{M^{\frac{1}{2}}} \to 0$ (e.g., $M \geq \kappa N^{2+\frac{2}{d}} \log N$ for some fixed $\kappa > 0$). 
%
For an $N$-point set of pairwise different points with $K$ pairs of antipodal points, we get
\begin{equation*}
S(m) = \frac{1}{N} + (-1)^m \, \frac{2K}{N^2} + \frac{1}{N^2} \mathop{\sum_{j=1}^N\sum_{k=1}^N}_{-1 < \PT{x}_j \cdot \PT{x}_k < 1} ( \PT{x}_j \cdot \PT{x}_k )^m  < \frac{1}{N} + (-1)^m \, \frac{2K}{N^2} + \left( \max_{\substack{1 \leq j,k \leq N \\ -1 < \PT{x}_j \cdot \PT{x}_k < 1}} \left| \PT{x}_j \cdot \PT{x}_k \right| \right)^m.
\end{equation*}
The double sum goes to zero as $m \to \infty$ and $N$ is fixed. We have the limit relations 
\begin{equation*}
\liminf_{m\to\infty} S(m) = \frac{1}{N} - \frac{2K}{N^2}, \qquad \limsup_{m\to\infty} S(m) = \frac{1}{N} + \frac{2K}{N^2}. 
\end{equation*}
For point sets that satisfy a well-separation property; i.e., $\| \PT{x}_j - \PT{x}_k \| \geq \frac{\beta}{N^{\frac{1}{d}}}$, $j \neq k$, for some positive $\beta$ not depending on $N$, we get 
\begin{equation*}
\sum_{m \geq \frac{2}{\beta^2} N^{\frac{2}{d}} \log N} \frac{-\Pochhsymb{-\frac{1}{2}}{m}}{m!} S(m) 
= \mathcal{O}\Big( \frac{1}{N^{1+\frac{1}{d}} \sqrt{\log N}} \Big) \qquad \text{as $N \to \infty$.}
\end{equation*}
We use that well-separation gives a bound on the number of points in a spherical cap centered at a point $\PT{x}_j$ and its antipodal position $-\PT{x}_j$. Hence, $\left| \PT{x}_j \cdot \PT{x}_k \right| \geq 1 - \frac{\beta^2}{2} \frac{1}{N^{\frac{2}{d}}}$ for at most $\beta^\prime N$ many pairs of points $( \PT{x}_j, \PT{x}_k )$ and therefore
\begin{align*}
\frac{1}{N^2} \sum_{j,k=1}^N ( \PT{x}_j \cdot \PT{x}_k )^{m} 
&\leq 
\frac{1}{N^2} \sum_{\substack{j,k=1 \\ | \PT{x}_j \cdot \PT{x}_k | \geq 1 - \frac{\beta^2}{2} \frac{1}{N^{\frac{2}{d}}}}}^N | \PT{x}_j \cdot \PT{x}_k |^{m} 
+ 
\frac{1}{N^2} \sum_{\substack{j,k=1 \\ | \PT{x}_j \cdot \PT{x}_k | < 1 - \frac{\beta^2}{2} \frac{1}{N^{\frac{2}{d}}}}}^N | \PT{x}_j \cdot \PT{x}_k |^{m} \\
&\leq 
\frac{1}{N^2} \sum_{\substack{j,k=1 \\ | \PT{x}_j \cdot \PT{x}_k | \geq 1 - \frac{\beta^2}{2} \frac{1}{N^{\frac{2}{d}}}}}^N 1  
+ 
\frac{1}{N^2} \sum_{\substack{j,k=1 \\ | \PT{x}_j \cdot \PT{x}_k | < 1 - \frac{\beta^2}{2} \frac{1}{N^{\frac{2}{d}}}}}^N \left( 1 -\frac{\beta^2}{2} \frac{1}{N^{\frac{2}{d}}} \right)^{m} \\
&\leq 
\frac{\beta^\prime}{N} 
+
e^{-\frac{\beta^2}{2} \frac{m}{N^{\frac{2}{d}}}} \\
&\leq 
\frac{\beta^\prime}{N} 
+
\frac{1}{N},
\end{align*}
provided $m \geq \frac{2}{\beta^2} N^{\frac{2}{d}} \log N$.
The subtractive part in $S(m)$ can be neglected by Lemma~\ref{lem:sum.asymptotics.alpha}. 

Numerical computations for so-called spherical Fibonacci points (images of Fibonacci lattice points in the unit square under an area-preserving Lambert transformation) conjectured to have nearly (possibly up to a logarithmic power) optimal spherical cap $\mathbb{L}_2$-discrepancy seem to exhibit all the previously discussed features; cf. Figure~\ref{fig:Fibonacci}. 

\begin{figure}[ht]
\centering
\includegraphics[scale=.2]{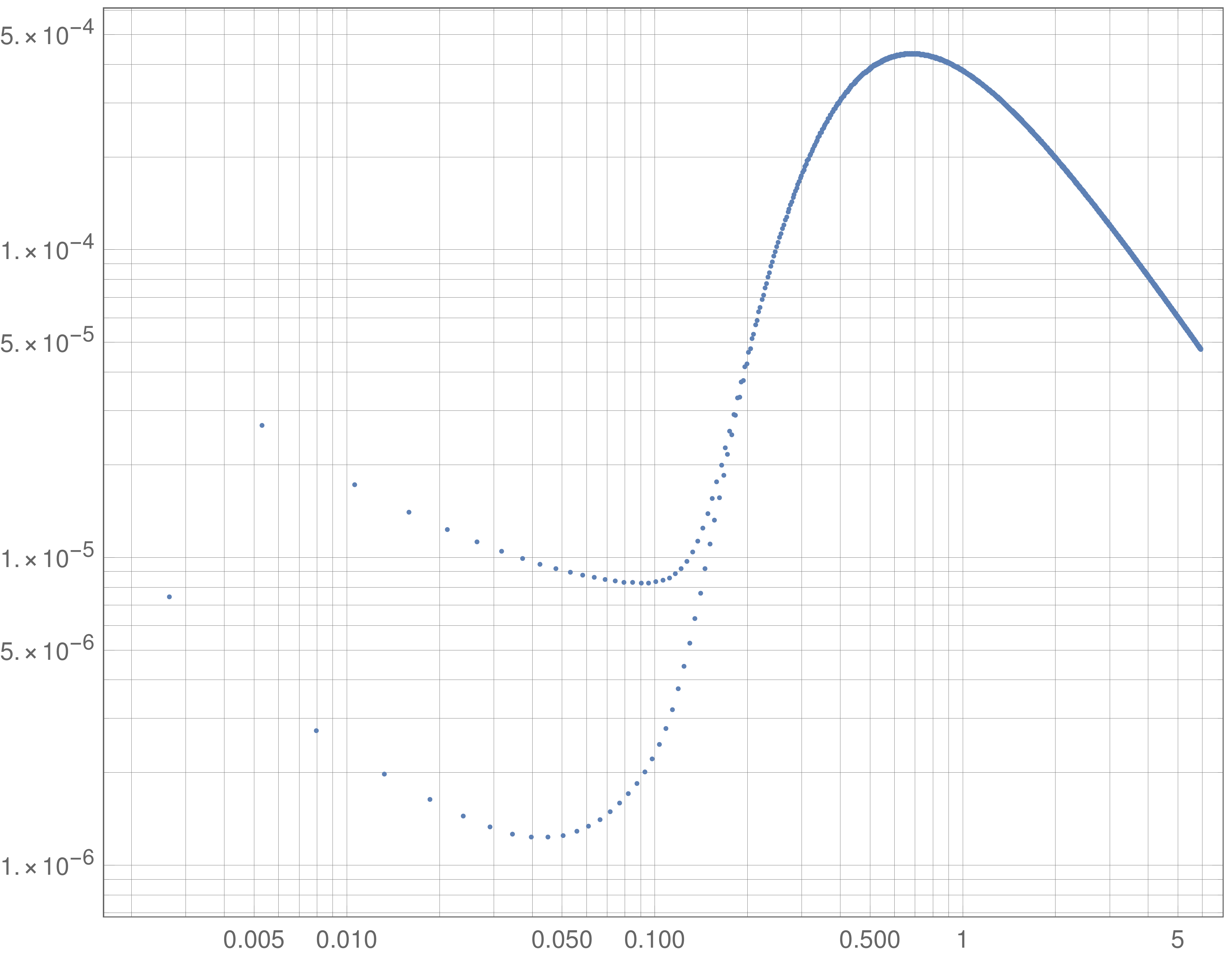} \\
\caption{\label{fig:Fibonacci} The behaviour of $N^{\frac{3}{2}}\frac{-\Pochhsymb{-\frac{1}{2}}{m}}{m!} S(m)$ versus $\frac{1}{N} \leq \frac{m}{N} \leq \log N$ for spherical Fibonacci points on $\mathbb{S}^2$ with $N = 377$ (i.e., $N = F_n$ for $n = 14$) points. There is different behaviour for even and odd $m$.}
\end{figure}

\section{Powers of the Euclidean distance and Generalizations}
\label{sec:powers.and.generalizations}

Since our argument exploits the sum of distances, rather than dealing with discrepancy directly,  it can also be utilized to analyze other positive power of the Euclidean distance. It has been proved in \cite{Wa1990:lower} that for all $0< \alpha < 2$ there exists a constant $\beta_{\alpha,d} > 0$ such that  for  any   $\PT{x}_1, \dots, \PT{x}_N  \in \mathbb S^d$ and $N$ sufficiently large 
\begin{equation}\label{eq:wagner}
\begin{split}
S_\alpha( \PT{x}_1, \dots, \PT{x}_N ) 
&\DEF \int_{\mathbb{S}^d} \int_{\mathbb{S}^d} \left\| \PT{x} - \PT{y} \right\|^\alpha \dd \sigma_d( \PT{x} ) \dd \sigma_d( \PT{y} ) - \frac{1}{N^2} \sum_{j=1}^N \sum_{k=1}^N \left\| \PT{x}_j - \PT{x}_k \right\|^\alpha \\
&\geq  \beta_{\alpha,d}  \, N^{-1-\frac{\alpha}{d}}.
\end{split}
\end{equation}
Existence of $N$-point configurations achieving an upper bound with the matching asymptotics in $N$ was proved in \cite{Wa1992:upper}. 

While the proof in \cite{Wa1990:lower} is quite technical and involved, we point out that our elementary argument presented in Section \ref{sec:simple.proof} applies to this situation almost verbatim.  The only modifications one needs to make consist of using the expansion
\begin{equation*}
\left( 1 - t \right)^{\frac{\alpha}{2}} = \sum_{m=0}^\infty (-1)^m \binom{\frac{\alpha}{2}}{m} t^m,
\end{equation*}
and observing that the coefficients $ (-1)^m \binom{\frac{\alpha}{2}}{m}$, $m \geq 1$, are positive and are of the order $m^{-1 - \frac{\alpha}{2}}$ as $m \to \infty$. Therefore the analogue of the sum  in \eqref{eq:sum} is of the order $M^{-\frac{\alpha}{2}} \asymp N^{-\frac{\alpha}{d}}$. Other than this, the proof of \eqref{eq:wagner} repeats the proof of \eqref{eq:Beck} verbatim. 

The argument for an asymptotic bound can be used in the same way. The analogue of~\eqref{eq:pre.bound.asymp.const} is
\begin{equation}
S_\alpha( \PT{x}_1, \dots, \PT{x}_N ) 
\geq 
2^{\frac{\alpha}{2}} \left( \frac{1}{N} \sum_{m=2M}^\infty \frac{-\Pochhsymb{-\frac{\alpha}{2}}{m}}{m!} - \sum_{r=M}^\infty \frac{-\Pochhsymb{-\frac{\alpha}{2}}{2r}}{(2r)!} \, \frac{\Pochhsymb{\frac{1}{2}}{r}}{\Pochhsymb{\frac{d+1}{2}}{r}} \right).
\end{equation}


By Lemma~\ref{lem:sum.asymptotics.alpha} and proceeding as in Subsection~\ref{subsec:asymptotic.constant}, we arrive at
\begin{equation*}
\liminf_{\substack{N\to\infty \\ \PT{x}_1, \dots, \PT{x}_N \in \mathbb{S}^d }} N^{1+\frac{\alpha}{d}} S_\alpha( \PT{x}_1, \dots, \PT{x}_N )
\geq 
c_{\alpha,d}^{\mathrm{asymp}},
\end{equation*}
where
\begin{equation} \label{eq:c.d.s.asymp}
c_{\alpha,d}^{\mathrm{asymp}} 
= \frac{1}{\gammafcn( 1 - \frac{\alpha}{2})} \, \frac{1}{1+\frac{\alpha}{d}} \left( \frac{2\gammafcn(\frac{1}{2})}{\gammafcn(\frac{d+1}{2})} \right)^{\frac{\alpha}{d}} = \left( \omega_d \right)^{\frac{\alpha}{d}} \frac{1}{\gammafcn( 1 - \frac{\alpha}{2})} \, \frac{1}{1+\frac{\alpha}{d}} \, \frac{1}{\pi^{\frac{\alpha}{2}}}.
\end{equation}
In particular, we get 
\begin{equation} \label{eq:c.2.s.asymp}
c_{\alpha,2}^{\mathrm{asymp}} 
= 
\frac{1}{\gammafcn( 1 - \frac{\alpha}{2})} \, \frac{2^\alpha}{1+\frac{\alpha}{2}}.
\end{equation}
We graphically compare our constant with the conjectured constant
\begin{equation} \label{eq:c.alpha.d.conj.via.lattices}
c_{\alpha,d}^{\mathrm{conj}} 
=
\left( \omega_d \right)^{\frac{\alpha}{d}} \left( - | \Lambda |^{-\frac{\alpha}{d}} \zetafcn_{\Lambda}\big( -\frac{\alpha}{2} \big) \right), 
\end{equation}
where $d$ is the dimension of the lattice $\lambda \in \{ \boldsymbol{A}_2, \boldsymbol{D}_4, \boldsymbol{E}_8, \boldsymbol{\Lambda}_{24} \}$ and $0\leq \alpha \leq 2$. 
This constant can be derived in a similar way as in Subsection~\ref{subsec:comparison} using a fundamental conjecture for the maximal sum of generalised distances given in \cite{BrHaSa2012b,BoHaSaBook2019}. 
Figure~\ref{fig:calphad} indicates that for each of the dimensions $d=2,4,8,16$, the conjectured value is larger than our asymptotic bound. The difference is small with a relative error that increases to about $25\%$ as $\alpha$ approaches~$2$ for $d=24$. 

\begin{figure}[ht]
\centering
\includegraphics[scale=.15]{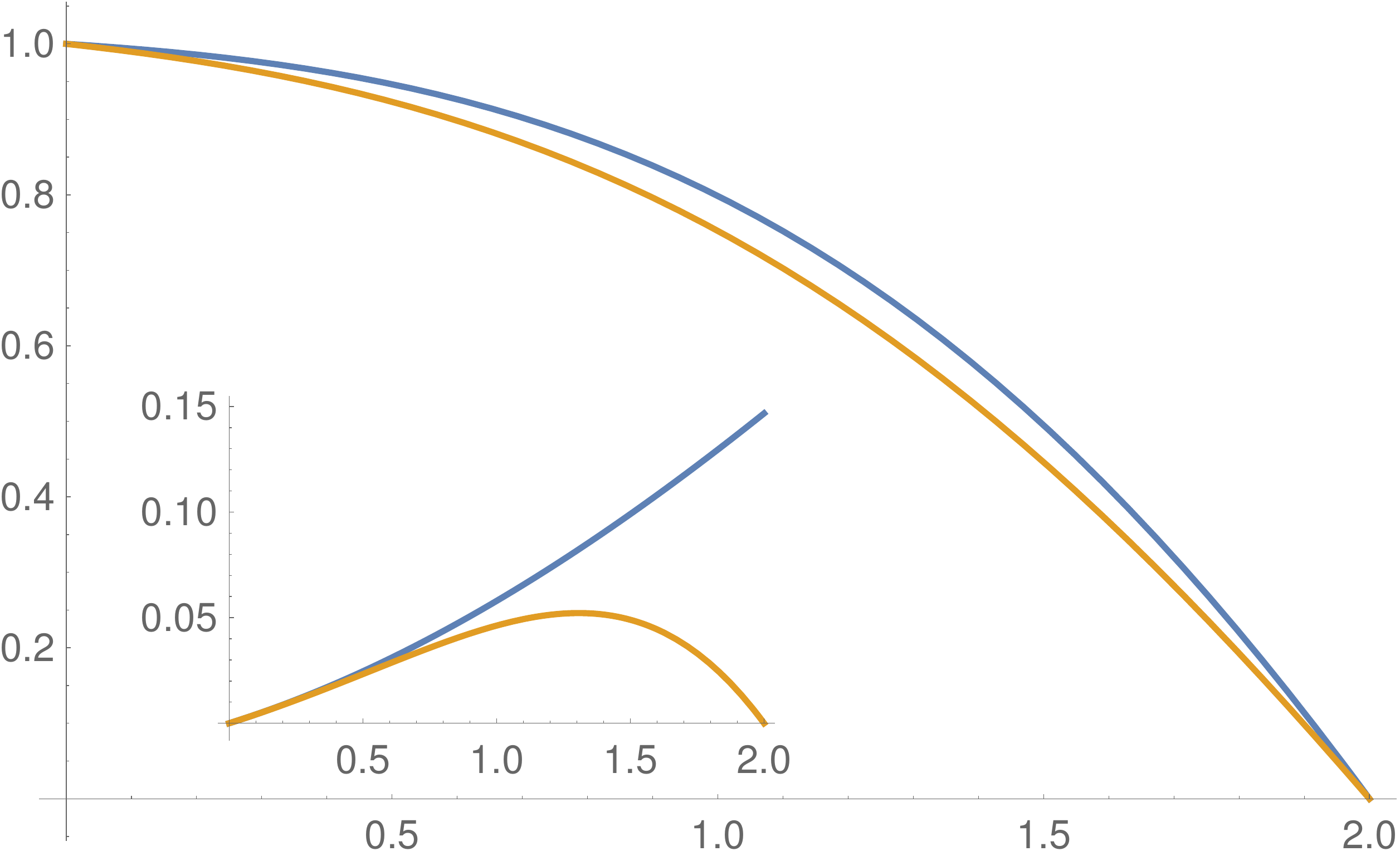} 
\includegraphics[scale=.15]{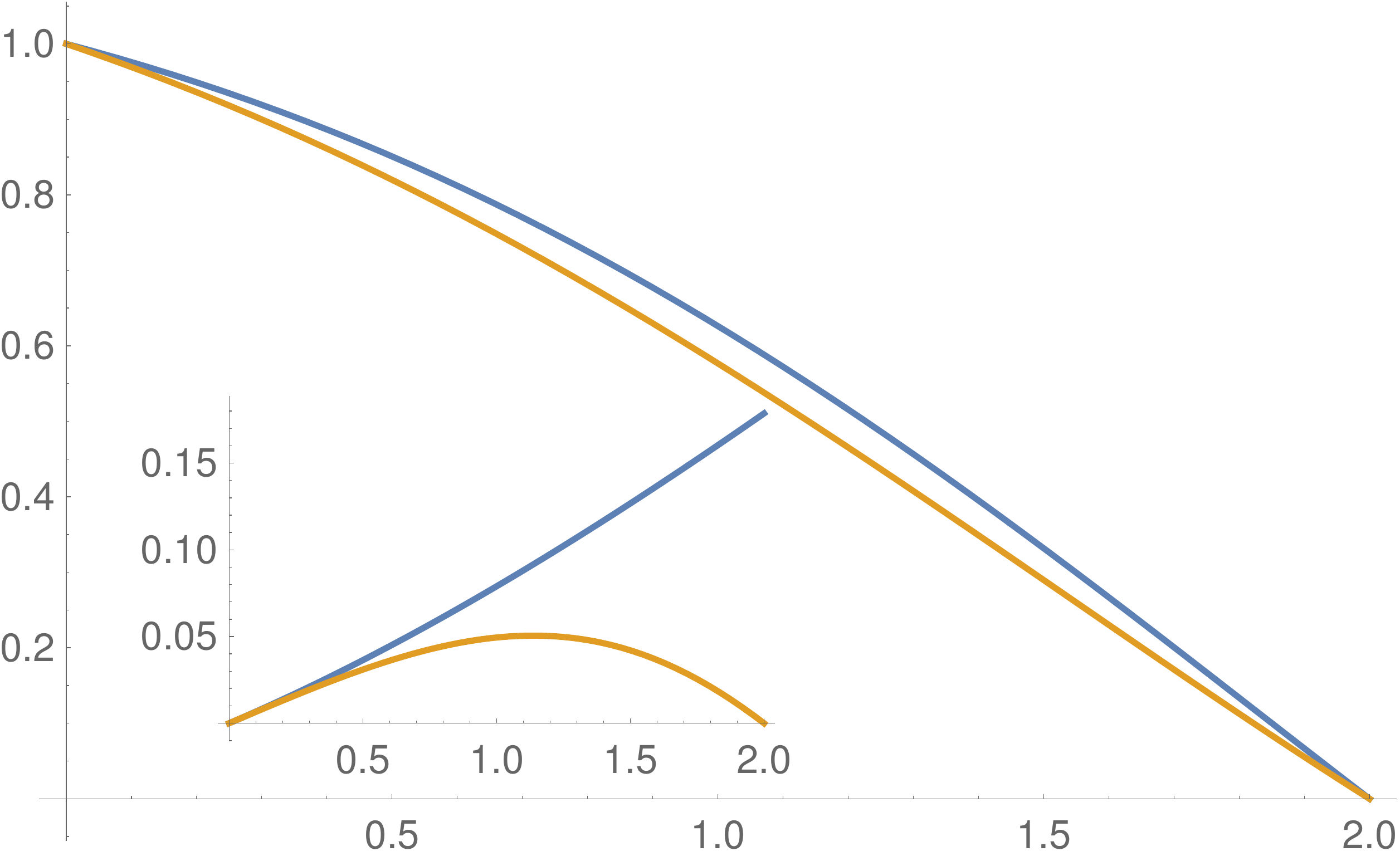}
\\
\includegraphics[scale=.15]{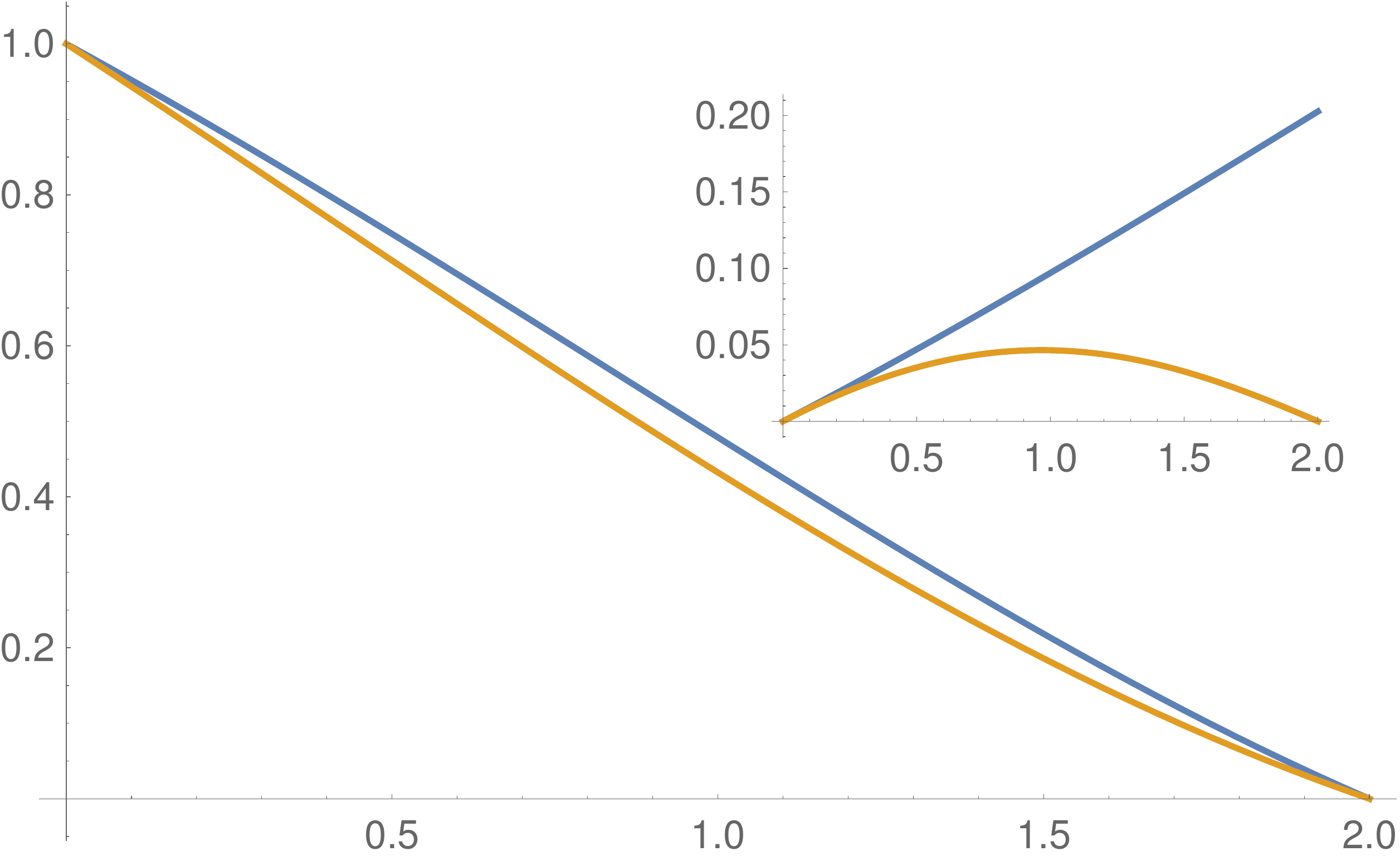} 
\includegraphics[scale=.15]{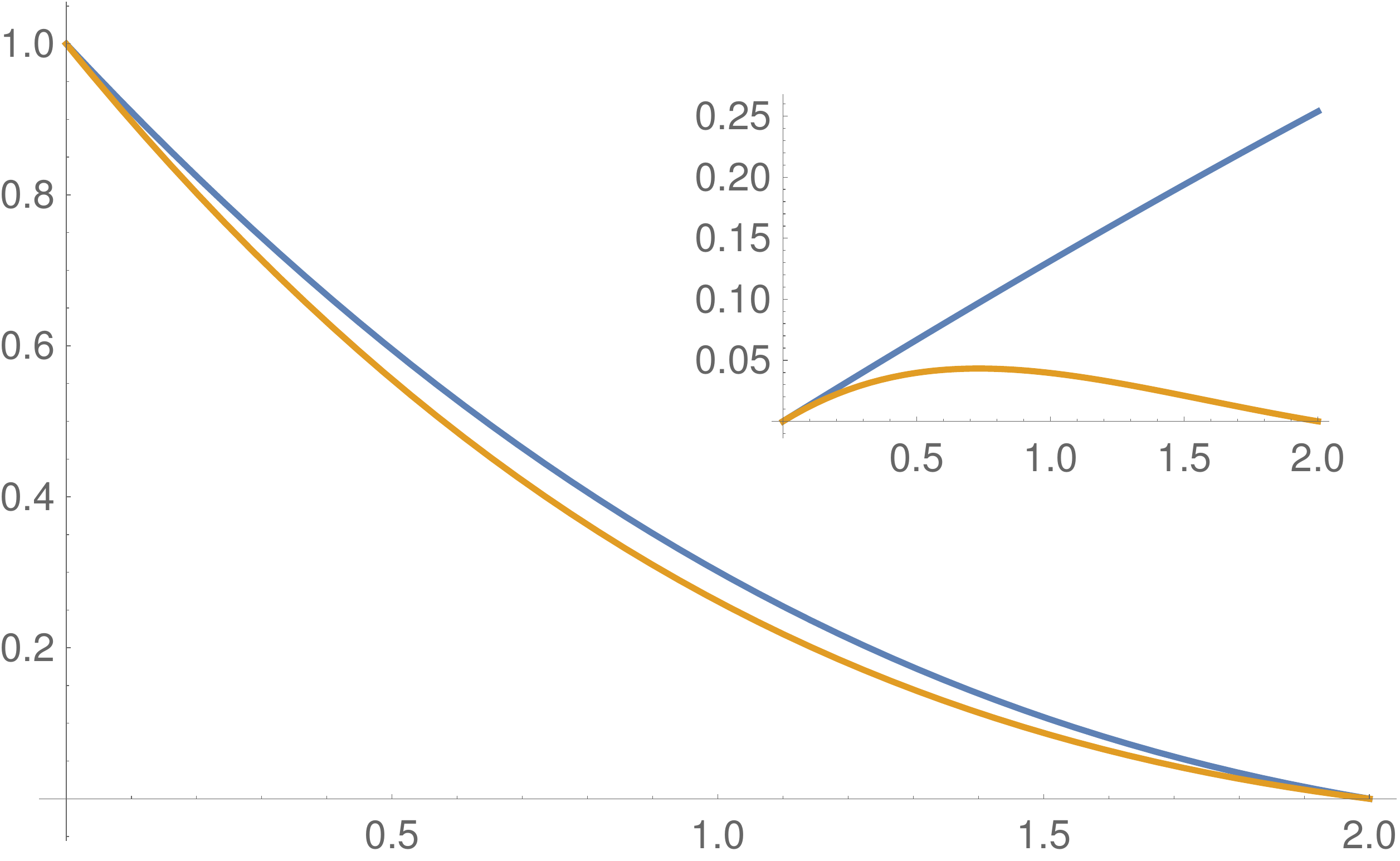}
\\
\caption{\label{fig:calphad} Comparison of the conjectured constant $c_{\alpha,d}^{\mathrm{conj}}$ with the  asymptotic constant $c_{\alpha,d}^{\mathrm{asymp}}$ for $0\leq \alpha \leq 2$ for $d = 2$ (top left), $d = 4$ (top right), $d = 8$ (bottom left), and $d = 16$ (bottom right). Inserts provide information about the increasing relative error $\frac{c_{\alpha,d}^{\mathrm{conj}}-c_{\alpha,d}^{\mathrm{asymp}}}{c_{\alpha,d}^{\mathrm{conj}}}$ and the error $c_{\alpha,d}^{\mathrm{conj}}-c_{\alpha,d}^{\mathrm{asymp}}$.}
\end{figure}

Indeed, one may even consider a general analytic  kernel 
\begin{equation}\label{eq:Kdef}
K( t ) \DEF \sum_{m=0}^\infty a_m t^m
\end{equation}
to arrive at a bound analogous to \eqref{eq:wagner}. 
\begin{prop}\label{prop:wagner+} 
Suppose, the coefficients $a_m$, $m \geq 1$, in \eqref{eq:Kdef} are positive and are of the order $m^{-1 - \frac{\alpha}{2}}$ as $m \to \infty$ for some fixed $\alpha >0$; i.e., $a_m m^{1 + \frac{\alpha}{2}}$  is uniformly bounded above and below for large $m$ (in particular, the Taylor series converges absolutely for $t\in[-1,1]$). Then there exists a constant $\gamma_{K,d} >0$ such that  for  any   $\PT{x}_1, \dots, \PT{x}_N  \in \mathbb S^d$ with $N$ large enough
\begin{equation}\label{eq:wagner+}
\frac{1}{N^2} \sum_{k,j=1}^N  K( \PT{x}_j \cdot \PT{x}_k ) - \int_{\mathbb{S}^d} \int_{\mathbb{S}^d}  K( \PT{x} \cdot \PT{y} )\, \dd \sigma_d( \PT{x} ) \dd \sigma_d( \PT{y} )  \ge  \gamma_{K,d} \, N^{-1-\frac{\alpha}{d}}.
\end{equation}
\end{prop}

For example, by \cite[Formula~7.3.2.84]{PrBrMa1990III} 
\begin{equation*}
K( t ) 
= 
\frac{1}{2}\sum_{m=0}^\infty \frac{\Pochhsymb{\frac{1}{2}}{m}}{\Pochhsymb{2}{m}} \, t^m 
=
\frac{1}{2}\Hypergeom{2}{1}{\frac{1}{2},1}{2}{t} 
=
\frac{1}{1 + \sqrt{1-t}}
\end{equation*}
with $\frac{1}{2}\frac{\Pochhsymb{\frac{1}{2}}{m}}{\Pochhsymb{2}{m}} = \frac{1}{2\sqrt{\pi}} \, \frac{\gammafcn( m + \frac{1}{2} )}{\gammafcn( m + 2 )} \sim \frac{1}{2\sqrt{\pi}} \, m^{-1-\frac{1}{2}}$ as $m \to \infty$.   Hence Proposition~\ref{prop:wagner+} applies with $\alpha =1$.  

Observe also that Proposition~\ref{prop:wagner+} continues to hold if the asymptotic relation $a_m \sim m^{-1 - \frac{\alpha}{2}}$ is imposed only on even indices $m$, since only even terms are used in  \eqref{eq:discrepancy.formula.lower.bound.01}.


\appendix

\section{Proof of Lemmas \ref{lem:1st.non.negativity} and \ref{lem:2nd.inner.product.integrals.asymptotics}}\label{app:lem}
We now turn to the proof of the two technical lemmas. 

\begin{proof}[Proof of Lemma \ref{lem:1st.non.negativity}]
 The case $m=0$ is obvious. 
 
Recall that the Gram matrix ${G_1 \DEF \big[ \PT{x}_j \cdot \PT{x}_k \big]_{j,k =1}^N}$ is positive semidefinite, and hence by Schur's product theorem, so is ${G_m \DEF \big[ \left( \PT{x}_j \cdot \PT{x}_k \right)^m \big]_{j,k =1}^N}$.

Since for each odd $m\ge 1$ the integral in \eqref{eq:pos} is zero, the statement immediately follows in this case: indeed, let $\PT{1}$ be a vector with $1$ in each coordinate, then 
\begin{equation*}
\sum_{j=1}^N \sum_{k=1}^N \left( \PT{x}_j \cdot \PT{x}_k \right)^m = \PT{1}^T G_m \PT{1} \geq 0.
\end{equation*}

In the general case, for any signed Borel measure $\mu$ on $\mathbb S^d$, define the energy integral 
\begin{equation*}
I_m (\mu ) \DEF \int_{\mathbb{S}^d} \int_{\mathbb{S}^d} \left( \PT{x} \cdot \PT{y} \right)^m  \dd \mu ( \PT{x} ) \dd \mu ( \PT{y} ).
\end{equation*}
Observe that positive semidefiniteness of $G_m$ implies that $I_m (\mu ) \ge 0$. Indeed, for a discrete measure $\mu = \sum c_j \delta_{\PT{x_j}}$, one has $I_m (\mu) = \sum_{j=1}^N \sum_{k=1}^N \left( \PT{x}_j \cdot \PT{x}_k \right)^m c_j c_k \ge 0$, and for general measures positivity follows by $w^*$-density of discrete measures. 

We also observe that for any Borel probability measure $\nu$ on $\mathbb S^d$, we have 
\begin{equation}\label{eq:diff}
0\le I_m (\nu - \sigma_d) = I_m (\nu)  + I_m (\sigma_d) - 2 \int_{\mathbb{S}^d}  \int_{\mathbb{S}^d} \left( \PT{x} \cdot \PT{y} \right)^m  \dd \sigma_d (\PT{x}) \dd \nu (\PT{y}) =   I_m (\nu)  - I_m (\sigma_d),
\end{equation} where we have used the fact that the inner integral with respect to $\dd \sigma_d (\PT{x})$ above is independent of $\PT{y}$ due to rotational invariance, and is thus equal to $I_m (\sigma_d)$. Taking $\nu = \frac{1}{N} \sum_{j=1}^N \delta_{\PT{x}_j}$ in \eqref{eq:diff} finishes the proof of the lemma since in this case $I_m (\nu) =  \frac{1}{N^2} \sum_{j,k=1}^N \left( \PT{x}_j \cdot \PT{x}_k \right)^m$. 
\end{proof}

We remark that the identity $I_m (\nu)  - I_m (\sigma_d) = I_m (\nu - \sigma_d)$ exhibited in \eqref{eq:diff} has been observed in, e.g., \cite{BoHaSaBook2019,La1972} and has been used in the context of discrepancy and Stolarsky principle in \cite{BiDa2019,BiDaMa2018}.

\begin{proof}[Proof of Lemma \ref{lem:2nd.inner.product.integrals.asymptotics}]
Due to rotational invariance, it suffices to compute the integral just with respect to $\PT{x}$. Since $\PT{x} \mapsto ( \PT{x} \cdot \PT{y} )^m$ is a zonal function, for odd integers $m$, the integral $\int_{\mathbb{S}^d} \left( \PT{x} \cdot \PT{y} \right)^m  \dd \sigma_d( \PT{x} )$ vanishes and for even integers $m = 2r$, the Funk-Hecke formula yields
\begin{align*}
\int_{\mathbb{S}^d} \left( \PT{x} \cdot \PT{y} \right)^m  \dd \sigma_d( \PT{x} )
&=
\frac{\omega_{d-1}}{\omega_d} \int_{-1}^1 t^{2r} \left( 1 - t^2 \right)^{\frac{d}{2}-1} \dd t 
= 
\frac{\omega_{d-1}}{\omega_d} \int_0^1 u^{r+\frac{1}{2}-1} \left( 1 - u \right)^{\frac{d}{2}-1} \dd u \\
&=
\frac{\omega_{d-1}}{\omega_d} \frac{\gammafcn( r + \frac{1}{2} ) \gammafcn( \frac{d}{2} )}{\gammafcn( r + \frac{d+1}{2} )} 
= 
\frac{\omega_{d-1}}{\omega_d} \frac{\gammafcn( \frac{d}{2} )\gammafcn( \frac{1}{2} )}{\gammafcn( \frac{d+1}{2} )} \, \frac{\Pochhsymb{\frac{1}{2}}{r}}{\Pochhsymb{\frac{d+1}{2}}{r}} 
= 
\frac{\Pochhsymb{\frac{1}{2}}{r}}{\Pochhsymb{\frac{d+1}{2}}{r}}.
\end{align*}
One can then easily obtain the asymptotics 
\begin{equation*}
\frac{\Pochhsymb{\frac{1}{2}}{r}}{\Pochhsymb{\frac{d+1}{2}}{r}} 
= 
\frac{\gammafcn( \frac{d+1}{2} )}{\gammafcn( \frac{1}{2} )} \, \frac{\gammafcn( r + \frac{1}{2} )}{\gammafcn( r + \frac{d+1}{2} )} 
\sim
\frac{\gammafcn( \frac{d+1}{2} )}{\gammafcn( \frac{1}{2} )} \, \frac{1}{r^{\frac{d}{2}}} \qquad \text{as $r \to \infty$.}
\end{equation*}
Application of Gautschi’s Inequality \cite[Eq.~5.6.4]{NIST:DLMF} gives
\begin{equation*}
\frac{\gammafcn( r + \frac{1}{2} )}{\gammafcn( r + \frac{d+1}{2} )} 
=
\begin{cases}
\dfrac{1}{\Pochhsymb{r+\frac{1}{2}}{\frac{d}{2}}} & \text{$d$ even,} \\
\dfrac{\gammafcn(r+\frac{1}{2})}{\gammafcn(r+1)} \, \dfrac{1}{\Pochhsymb{r+1}{\frac{d-1}{2}}} & \text{$d$ odd}
\end{cases}
<
\frac{1}{r^{\frac{d}{2}}}.
\end{equation*}
\end{proof}

\section{Proof of Lemma~\ref{lem:sum.asymptotics.alpha}} 
\label{app:proof.lem:sum.asymptotics.alpha}

\begin{proof}
Let $0 < \alpha < 2$. With the help of Mathematica
\begin{equation*}
\sum_{m=2M}^\infty \frac{-\Pochhsymb{-\frac{\alpha}{2}}{m}}{m!}
=
\frac{1}{\gammafcn( 1 - \frac{\alpha}{2} )} \, \frac{\gammafcn( 2M - \frac{\alpha}{2} )}{\gammafcn( 2 M )}.
\end{equation*}
The asymptotic relation follows, e.g., from \cite[Eq.~5.11.13]{NIST:DLMF}. 

The second series can be written as a $3F2$-hypergeometric function,
\begin{equation*}
\sum_{r=M}^\infty \frac{-\Pochhsymb{-\frac{\alpha}{2}}{2r}}{(2r)!} \, \frac{\Pochhsymb{\frac{1}{2}}{r}}{\Pochhsymb{\frac{d+1}{2}}{r}}
=
\frac{-\Pochhsymb{-\frac{\alpha}{2}}{2M}}{(2M)!} \, \frac{\Pochhsymb{\frac{1}{2}}{M}}{\Pochhsymb{\frac{d+1}{2}}{M}} \Hypergeom{3}{2}{M-\frac{\alpha}{4}, M-\frac{\alpha}{4}+\frac{1}{2}, 1}{M+\frac{d+1}{2},M+1}{1}.
\end{equation*}
Since $\frac{d+\alpha}{2} > 0$ and $M + \frac{d-1}{2} > 0$ for $M \geq 1$, we may apply \cite[Eq.~7.4.4.1, p.~533]{PrBrMa1990III} to obtain
\begin{equation*} 
\sum_{r=M}^\infty \frac{-\Pochhsymb{-\frac{\alpha}{2}}{2r}}{(2r)!} \, \frac{\Pochhsymb{\frac{1}{2}}{r}}{\Pochhsymb{\frac{d+1}{2}}{r}}
=
\frac{-\Pochhsymb{-\frac{\alpha}{2}}{2M}}{(2M)!} \, \frac{\Pochhsymb{\frac{1}{2}}{M}}{\Pochhsymb{\frac{d+1}{2}}{M}} \, \frac{2M + d - 1}{d+\alpha} \Hypergeom{3}{2}{\frac{1}{2}+\frac{\alpha}{4}, 1+\frac{\alpha}{4}, 1}{1+\frac{d+\alpha}{2},M+1}{1}.
\end{equation*}
The $3F2$-hypergeometric function has a convergent series expansion of positive terms which can be seen as a generalized asymptotic expansion as $M \to \infty$. By \eqref{eq:Pochhammer.identities} and simplification, 
\begin{equation*}
\begin{split}
&\sum_{r=M}^\infty \frac{-\Pochhsymb{-\frac{\alpha}{2}}{2r}}{(2r)!} \, \frac{\Pochhsymb{\frac{1}{2}}{r}}{\Pochhsymb{\frac{d+1}{2}}{r}} \\
&\phantom{equ}=
2^{-\frac{\alpha}{2}} \frac{\alpha \gammafcn( \frac{d+1}{2} )}{2 \sqrt{\pi} \gammafcn( 1 - \frac{\alpha}{2} ) ( d + 1 )} \, \frac{\gammafcn( M-\frac{\alpha}{4} ) \gammafcn( M-\frac{\alpha}{4}+\frac{1}{2} )}{\gammafcn( M+1 ) \gammafcn( M+\frac{d-1}{2} )} \Hypergeom{3}{2}{\frac{1}{2}+\frac{\alpha}{4}, 1+\frac{\alpha}{4}, 1}{1+\frac{d+\alpha}{2},M+1}{1}.
\end{split}
\end{equation*}
Using asymptotic expansions of ratios of gamma functions for large $M$ and the series expansion of the hypergeometric function, we arrive at the desired result. 
\end{proof}

\bibliographystyle{abbrv}
\bibliography{REFS}

\begin{thebibliography}{10}

\bibitem{AiBrDi2012}
C.~Aistleitner, J.~S. Brauchart, and J.~Dick.
\newblock Point sets on the sphere {$\mathbb{S}^2$} with small spherical cap
  discrepancy.
\newblock {\em Discrete Comput. Geom.}, 48(4):990--1024, 2012.

\bibitem{AlZa2015}
K.~Alishahi and M.~Zamani.
\newblock The spherical ensemble and uniform distribution of points on the
  sphere.
\newblock {\em Electron. J. Probab}, 20(23):1--27, 2015.

\bibitem{Be1984}
J.~Beck.
\newblock Sums of distances between points on a sphere---an application of the
  theory of irregularities of distribution to discrete geometry.
\newblock {\em Mathematika}, 31(1):33--41, 1984.

\bibitem{BeFi2003}
J.~J. Benedetto and M.~Fickus.
\newblock Finite normalized tight frames.
\newblock {\em Advances in Computational Mathematics}, 18:357--385, 2003.

\bibitem{BiDa2019}
D.~Bilyk and F.~Dai.
\newblock Geodesic distance {R}iesz energy on the sphere.
\newblock {\em Trans. Amer. Math. Soc.}, 372(5):3141--3166, 2019.

\bibitem{BiDaMa2018}
D.~Bilyk, F.~Dai, and R.~Matzke.
\newblock The {S}tolarsky principle and energy optimization on the sphere.
\newblock {\em Constr. Approx.}, 48(1):31--60, 2018.

\bibitem{BoRaVi2013}
A.~Bondarenko, D.~Radchenko, and M.~Viazovska.
\newblock Optimal asymptotic bounds for spherical designs.
\newblock {\em Ann. of Math. (2)}, 178(2):443--452, 2013.

\bibitem{BoGrMa2024}
B.~Borda, P.~Grabner, and R.~W. Matzke.
\newblock {Riesz energy, $L^2$ discrepancy, and optimal transport of
  determinantal point processes on the sphere and the flat torus}.
\newblock {\em Mathematika}, 70(2):e12245, 2024.

\bibitem{BoHaSaBook2019}
S.~V. Borodachov, D.~P. Hardin, and E.~B. Saff.
\newblock {\em Discrete energy on rectifiable sets}.
\newblock Springer Monographs in Mathematics. Springer, New York, [2019]
  \copyright 2019.

\bibitem{Br2011}
J.~S. Brauchart.
\newblock Optimal discrete {R}iesz energy and discrepancy.
\newblock {\em Unif. Distrib. Theory}, 6(2):207--220, 2011.

\bibitem{BrDi2013}
J.~S. Brauchart and J.~Dick.
\newblock A simple proof of {S}tolarsky's invariance principle.
\newblock {\em Proc. Amer. Math. Soc.}, 141(6):2085--2096, 2013.

\bibitem{BrGrKu2019}
J.~S. Brauchart, P.~J. Grabner, and W.~Kusner.
\newblock Hyperuniform point sets on the sphere: deterministic aspects.
\newblock {\em Constr. Approx.}, 50(1):45--61, 2019.

\bibitem{BrHaSa2012b}
J.~S. Brauchart, D.~P. Hardin, and E.~B. Saff.
\newblock The next-order term for optimal {R}iesz and logarithmic energy
  asymptotics on the sphere.
\newblock In {\em Recent advances in orthogonal polynomials, special functions,
  and their applications}, volume 578 of {\em Contemp. Math.}, pages 31--61.
  Amer. Math. Soc., Providence, RI, 2012.

\bibitem{BrSaSlWo2014}
J.~S. Brauchart, E.~B. Saff, I.~H. Sloan, and R.~S. Womersley.
\newblock Q{MC} designs: {O}ptimal order {Q}uasi {M}onte {C}arlo integration
  schemes on the sphere.
\newblock {\em Math. Comp.}, 83(290):2821--2851, 2014.

\bibitem{CoCoElKu2007}
H.~Cohn, J.~H. Conway, N.~D. Elkies, and A.~Kumar.
\newblock The {$D_4$} root system is not universally optimal.
\newblock {\em Experiment. Math.}, 16(3):313--320, 2007.

\bibitem{CoKu2007}
H.~Cohn and A.~Kumar.
\newblock Universally optimal distribution of points on spheres.
\newblock {\em J. Amer. Math. Soc.}, 20(1):99--148, 2007.

\bibitem{CoKuMiRaVi2022}
H.~Cohn, A.~Kumar, S.~D. Miller, D.~Radchenko, and M.~Viazovska.
\newblock Universal optimality of the {$E_8$} and {L}eech lattices and
  interpolation formulas.
\newblock {\em Ann. of Math. (2)}, 196(3):983--1082, 2022.

\bibitem{CoSl1988}
J.~H. Conway and N.~J.~A. Sloane.
\newblock {\em Sphere packings, lattices and groups}, volume 290 of {\em
  Grundlehren der mathematischen Wissenschaften [Fundamental Principles of
  Mathematical Sciences]}.
\newblock Springer-Verlag, New York, 1988.
\newblock With contributions by E. Bannai, J. Leech, S. P. Norton, A. M.
  Odlyzko, R. A. Parker, L. Queen and B. B. Venkov.

\bibitem{NIST:DLMF}
{\it NIST Digital Library of Mathematical Functions}.
\newblock \url{https://dlmf.nist.gov/}, Release 1.2.0 of 2024-03-15.
\newblock F.~W.~J. Olver, A.~B. {Olde Daalhuis}, D.~W. Lozier, B.~I. Schneider,
  R.~F. Boisvert, C.~W. Clark, B.~R. Miller, B.~V. Saunders, H.~S. Cohl, and
  M.~A. McClain, eds.

\bibitem{EhGr}
M.~Ehler and K.~Gr\"ochenig.
\newblock {$t$}-design curves and mobile sampling on the sphere.
\newblock {\em Forum Math. Sigma}, 11:Paper No. e105, 25, 2023.

\bibitem{Et2021}
U.~Etayo.
\newblock {Spherical Cap Discrepancy of the Diamond Ensemble}.
\newblock {\em Discrete Comput Geom}, 66:1218--1238, 2021.

\bibitem{La1972}
N.~S. Landkof.
\newblock {\em Foundations of modern potential theory}.
\newblock Springer-Verlag, New York, 1972.
\newblock Translated from the Russian by A. P. Doohovskoy, Die Grundlehren der
  mathematischen Wissenschaften, Band 180.

\bibitem{PrBrMa1990III}
A.~P. Prudnikov, Y.~A. Brychkov, and O.~I. Marichev.
\newblock {\em Integrals and series. {V}ol. 3}.
\newblock Gordon and Breach Science Publishers, New York, 1990.
\newblock More special functions, Translated from the Russian by G. G. Gould.

\bibitem{SaSt2006}
P.~Sarnak and A.~Str\"{o}mbergsson.
\newblock Minima of {E}pstein's zeta function and heights of flat tori.
\newblock {\em Invent. Math.}, 165(1):115--151, 2006.

\bibitem{Si}
V.~M. Sidel’nikov. 
\newblock New estimates for the closest packing of spheres in n-dimensional Euclidean space. 
\newblock {\em Mat. Sb.}, 24:148--158, 1974. 


\bibitem{St1973}
K.~B. Stolarsky.
\newblock Sums of distances between points on a sphere. {II}.
\newblock {\em Proc. Amer. Math. Soc.}, 41:575--582, 1973.

\bibitem{Wa1990:lower}
G.~Wagner.
\newblock On means of distances on the surface of a sphere (lower bounds).
\newblock {\em Pacific J. Math.}, 144(2):389--398, 1990.

\bibitem{Wa1992:upper}
G.~Wagner.
\newblock On means of distances on the surface of a sphere. {II}. {U}pper
  bounds.
\newblock {\em Pacific J. Math.}, 154(2):381--396, 1992.

\bibitem{Welch}
L.~Welch.
\newblock Lower bounds on the maximum cross correlation of signals (corresp.).
\newblock {\em IEEE Transactions on Information Theory}, 20(3):397--399, 1974.

\end{thebibliography}

\end{document}